\documentclass{amsart}

\usepackage[utf8]{inputenc}
\usepackage[margin=1in]{geometry}
\usepackage{amsmath, amssymb, amsthm}
\usepackage{dsfont}
\usepackage{xcolor}
\usepackage{mathtools}
\usepackage{enumitem}
\usepackage{hyperref}

\usepackage{tikz}
\usetikzlibrary{decorations.pathmorphing}
\usetikzlibrary{calc}

\makeatletter
\newenvironment{customproof}[1]{%
  \par\pushQED{\qed}%
  \normalfont\topsep6pt \trivlist
  \item[\hskip\labelsep\itshape
    Proof of #1\@addpunct{.}]\ignorespaces
}{%
  \popQED\endtrivlist\@endpefalse
}
\makeatother

\theoremstyle{plain}
\newtheorem{theorem}{Theorem}[section]
\newtheorem{lemma}[theorem]{Lemma}
\newtheorem{proposition}[theorem]{Proposition}
\newtheorem{corollary}[theorem]{Corollary}

\newtheorem{preremark}[theorem]{Remark}
\newenvironment{remark}{\begin{preremark}\rm}{\medskip \end{preremark}}

\theoremstyle{definition}
\newtheorem{definition}[theorem]{Definition}

\numberwithin{equation}{section}

\newcommand{\R}{\mathbb R}
\newcommand{\LL}{\mathcal L}
\newcommand{\A}{\mathcal A}
\newcommand{\dd}{\, \mathrm{d}}
\newcommand{\eps}{\varepsilon}
\newcommand{\one}{\mathds{1}}
\DeclareMathOperator{\supp}{supp}
\DeclareMathOperator{\dist}{dist}
\newcommand{\p}{\partial}
\DeclareMathOperator*{\Tail}{{Tail}}
\newcommand{\dsp}{{(-\Delta_p)^s}}

\newcommand{\gs}{\gamma}
\DeclareMathOperator{\osc}{osc}
\newcommand{\kmax}{{N^*}}

\title{H\"older gradient estimates for fractional $p$-caloric functions}
\author{Davide Giovagnoli, David Jesus, Luis Silvestre}

\newcommand{\Addresses}{{
  \bigskip
  \footnotesize

    Davide Giovagnoli, \textsc{Dipartimento di Matematica,
Universit\`a di Bologna,
 Bologna, 40126, Italy}\par\nopagebreak
  \textit{E-mail address:} {\tt d.giovagnoli@unibo.it}\\

  \medskip

    David Jesus, \textsc{Dipartimento SMFI,
Universit\'a di Parma,
Parco Area delle Scienze 53/a, Parma, 43124, Italy}\par\nopagebreak
  \textit{E-mail address:} {\tt djbj1993@gmail.com}\\

\medskip

  Luis Silvestre, \textsc{Mathematics Department,
 University of Chicago,
 Chicago, IL 60637, USA}\par\nopagebreak
  \textit{E-mail address:} {\tt luis@math.uchicago.edu}\\

 }

}

\begin{document}

\begin{abstract}

We establish interior $C^{1,\alpha}$ regularity estimates, for some $\alpha > 0$, for fractional $p$-caloric functions when $p$ is in the range $p \in (2,2/(1-s))$.  As a consequence, we deduce improved regularity in time, which yields interior continuous differentiability in time for $p \in [1/(1-s), 2/(1-s))$.

\end{abstract}

\maketitle

\section{Introduction}\label{s:introduction}

The parabolic fractional $p$-Laplace equation
\begin{align}\label{eq:fracplap}
    \p_tu+(-\Delta_p)^su=0 \quad
\end{align}
is, up to a constant, the $L^2$ gradient flow of the Gagliardo $W^{s,p}$ seminorm for $s \in (0,1)$ and $p \in [1,\infty)$, which is defined by
\begin{align*}
    [u]_{W^{s,p}(\R^d)} := \left( \iint_{\R^d \times \R^d} \frac{|u(x)-u(y)|^p}{|x-y|^{d+sp}} \dd x \dd y \right)^{1/p}.
\end{align*}
The local counterpart of \eqref{eq:fracplap} is the parabolic $p$-Laplace equation for which H\"older gradient estimates were obtained in \cite{DiBenedettoFriedman1985}, see also \cite{wiegner1986385,Imbert-Jin-Silvestre}. 

The fractional $p$-Laplace operator that appears in \eqref{eq:fracplap} is given by the principal-value integral
\[
(-\Delta_p)^s u(t,x) := \mathrm{P.V.} \int_{\R^d} \frac{|u(t,x)-u(t,y)|^{p-2}}{|x-y|^{d+sp}} (u(t,x)-u(t,y)) \dd y,
\]
acting on the spatial variable. It is the parabolic counterpart of the (elliptic) fractional $p$-Laplace equation $(-\Delta_p)^s u =0$, which has been the subject of intense study during the last decade.

For the elliptic equation, $C^{1,\alpha}$ regularity in the range $p \in [2,2/(1-s))$ was recently established in \cite{GJS} after Lipschitz continuity had been obtained in \cite{biswas2025lipschitz,biswas2025improved}. An overview of the regularity results for the stationary case, along with the references therein, can be found in \cite[Section 1.1]{GJS}. We provide here a non-exhaustive list of related works \cite{ishii2010class,di2014nonlocal,kuusi2015nonlocal,di2016local,Iannizzotto2016global,cozzi2017regularity,brasco2017higher,brasco2018higher,korvenpaa2019equivalence,bogelein2024regularity,iannizzotto2024fine,bogelein2025gradient,Diening2025higher,Diening2025calderon,iannizzotto2026boundary}. 
By contrast, the regularity theory for the parabolic equation \eqref{eq:fracplap} has been more limited. Local boundedness of weak solutions was obtained in \cite{stromqvist2019local}. 
The H\"older continuity was proved in a series of works \cite{BLS_Para,LiaoParabolic,Garain_Lindgren_Tavakoli}. Recently, in \cite{JSU}, Jesus, Sobral and Urbano established Lipschitz continuity in the spatial variable for solutions of \eqref{eq:fracplap} in the range $p \in (2,2/(1-s))$.  In the present paper we use that result as a starting point and prove that solutions of \eqref{eq:fracplap} are $C^{1,\alpha}$ regular in the range of parameters $p \in [2,2/(1-s))$.

The main result of the paper is the following. 
Here and throughout the work, we use the notation $Q_{r_1,r_2}(t_0,x_0)$ to denote the following cylinder in space-time
\[ Q_{r_1,r_2}(t_0,x_0) := (t_0 - r_1, t_0] \times B_{r_2}(x_0). \]
When we write $Q_{r_1,r_2}$ we mean $Q_{r_1,r_2}(0,0)$. 

\begin{theorem} \label{t:main}
    Assume that $s \in (0,1)$ and $2 < p <2/(1-s)$. Let $u \in C(-2,0;L_{sp}^{p-1}(\R^d))$ be a weak solution of \eqref{eq:fracplap} in $Q_{2,2}$.
     Then $\nabla u$ exists pointwise in $Q_{1,1}$ and the following estimate holds, for some $\alpha > 0$. Let
\begin{equation} \label{def:data}
   M := 1 + \|u\|_{L^\infty(Q_{2,2})} + \sup_{t\in(-2,0]} {\Tail}_{p-1,sp}(u(t,\cdot);2)<\infty.
\end{equation} 
Then, for all $(t,x),(\tau,y) \in Q_{1,1}$,
  \[ |\nabla u(t,x) - \nabla u(\tau,y)| \leq C M\left( M^{(p-2)\alpha} |t-\tau|^\alpha + |x-y|^\alpha  \right). \]
Here $\alpha$ and $C$ are positive constants depending only on the dimension $d$, $s$ and  $p$. The quantity $ {\Tail}_{p-1,sp}(u(t,\cdot);2)$ accounts for the values of $u(t,\cdot)$ outside the ball $B_2$. It is defined in Section \ref{s:tails}.
\end{theorem}

We also prove that solutions are locally continuously differentiable in time in the range $1/(1-s) \leq p <2/(1-s)$, $p>2$.
\begin{theorem}\label{t:time_reg}
    Assume that $s \in (0,1)$ and $1/(1-s) \leq p <2/(1-s)$, $p>2$. Let $u \in C(-2,0;L_{sp}^{p-1}(\R^d))$ be a weak solution of \eqref{eq:fracplap} in $Q_{2,2}$.
     Then $u$ is locally $C^1$ in time.
\end{theorem}

Note that our result does not contradict the counterexample in \cite[Remark 1.3]{BLS_Para}. Since we are assuming the tail is continuous and there the authors obtain the discontinuity of $\p_t u$ using exactly the discontinuity in time of the tail.

In fact, Proposition~\ref{p:time-regularity} provides a finer picture of the time-regularity regimes, improving the known regularity in time from \cite{JSU} (see Theorem \ref{t:JSU-Lip} below) in the following ways.
Let $\alpha$ be the exponent from Theorem \ref{t:main}. For  $ \frac{1-\alpha}{1-s} < p < \frac{2}{1-s}$, $p>2$  the solution is $C^1$ in time, if $p=\frac{1-\alpha}{1-s}$ the solution is $\log$-Lipschitz and for $2<p<\frac{1-\alpha}{1-s}$ the solution is H\"older with exponent $\frac{1+\alpha}{2-p(1-s)}$.

Theorems~\ref{t:main} and \ref{t:time_reg} apply to weak solutions of \eqref{eq:fracplap} which, as proven in \cite{JSU}, are also viscosity solutions. The reverse implication was only proven under the stronger assumption $u\in L^\infty((-1,0]\times \R^d)$. The definitions of weak and viscosity solutions of \eqref{eq:fracplap} are given in Section \ref{s:weak/viscosity}.

We use the following result as the starting point of our analysis.

\begin{theorem}[\protect{\cite[Theorem 1.1]{JSU}}] \label{t:JSU-Lip}
Let $s \in (0,1)$ and $2< p < 2/(1-s)$. Let $u \in C(-2,0;L_{sp}^{p-1}(\R^d))$ be a weak solution of \eqref{eq:fracplap} in $Q_{2,2}$
 and set $M$ as in \eqref{def:data}.
 Then $u$ is $C^{0,1}$ in the spatial variables and $C^{0,\beta}$ in the time variable in $Q_{1,1}$ where $\beta:=\min\{ 1,\left(\frac{1}{2-p(1-s)}\right)^-\}$. Moreover, the following estimate holds
\[
|u(t,x)-u(\tau,y)| \leq CM \left(M^{(p-2)\beta}|t-\tau|^{\beta} +|x-y| \right)
\]
for all $(t,x),(\tau,y) \in Q_{1,1}$.
\end{theorem}

\subsection{Strategy of the proof} \label{s:strategy}

Our strategy to prove Theorem~\ref{t:main} follows the same steps as that of the elliptic paper \cite{GJS}, with several adaptations to handle the time variable. The main difference is in the proof of the \emph{first stage}, which is essentially different to the proof in the elliptic case. This is the most challenging part of the proof. We explain the ideas and difficulties below, for every part of the proof.

By Theorem~\ref{t:JSU-Lip}, the spatial gradient $\nabla u$ is bounded in any cylinder strictly inside $Q_{2,2}$. Up to a rescaling, we suppose $\|\nabla u\|_{L^\infty(Q_{1,1})} \leq 1$. Our goal is to prove that $\nabla u$ is H\"older continuous at the origin. The proof proceeds in three stages. 

The \textbf{first stage} establishes an iterative improvement of the
gradient bound at geometrically decreasing two-parameter scales
\[
\|\nabla u\|_{L^\infty(Q_{\rho_1^k,\rho_2^k})}
\leq C_0(1-\delta)^k,
\qquad k=0,1,\dots,\kmax.
\]
The necessity of using two different scaling parameters is dictated by
the natural scaling of the equation; we refer to
Subsection~\ref{ss:scaling}. The iteration stops when, at one of these
scales, the gradient becomes close, in most of the corresponding
cylinder, to a fixed direction \(e\in S^{d-1}\).

The proof of the one-step improvement is carried out in
Section~\ref{s:reduce_norm}. It is not merely a parabolic
adaptation of the elliptic counterpart in \cite{GJS}. The time variable introduces two genuine difficulties.
First, the information that the gradient stays away from \(e\) is
distributed over time and must be transferred to the later time
interval on which the improvement is required. To accomplish this, we
introduce time-dependent heights governed by an ordinary differential
equation whose forcing records the visible spatial mass of the set
where the gradient is away from \(e\). Second, the elliptic one-step
measure reduction cannot be applied independently on each time slice.
We replace it with a space--time mass-transfer argument of
Caffarelli--Chan--Vasseur type \cite{chihinchan}. Since the linearized kernel may
degenerate, the usual uniform coercivity assumption is replaced by a
quantitative lower bound for the gap term in the linearized energy
inequality, obtained from coercivity along suitable cones of
nondegenerate directions.

A further difficulty is controlling the oscillation in time under rescaling. In Section~\ref{s:oscillation_time}, we estimate the time oscillation after subtracting $u(t,0)$. In Section~\ref{s:iteration}, we iterate the improvement in the cylinders $Q_{\rho_1^k,\rho_2^k}$ and propagate the bound on the tail. 

The \textbf{second stage} uses the conclusion of the first stage: in most of a small cylinder, the gradient $\nabla u$ is close to a fixed vector $e$. We deduce that $\nabla u$ is close to $e$ at every point of a smaller cylinder. As in the elliptic case \cite{GJS,biswas2025lipschitz,biswas2025improved}, this is done with the Ishii-Lions method applied to the function $u(t,x) - e \cdot x$. The implementation is similar to the one in the parabolic paper \cite{JSU}, with the additional control coming from the small oscillation of $u(t,x) - e \cdot x$ in the cylinder.

In the \textbf{third stage} we use the linearized equation again, but this time in a regime in which $\nabla u$ stays close to a fixed nonzero vector $e$. The kernel of the linearized operator is then mildly degenerate and we are essentially in the realm of integro-differential operators with positive bounds on the kernel, in the spirit of \cite{Kassman-Schwab,FelsingerKassmann}. This last stage produces the H\"older estimate for $\nabla u$ in space and time.

The three stages combine to give a $C^{1,\alpha}$ estimate at the origin, in the parabolic sense. The conclusion of Theorem~\ref{t:main} follows by translating this pointwise estimate.

 The bulk of the work is in the first stage, which is where the genuinely new ideas of \cite{GJS} need to be reworked in the parabolic setting.

\subsection{Notation}

We collect here the basic notation used throughout the paper. We write $B_R(x_0)$ for the open ball in $\R^d$ of radius $R$ centered at $x_0$, and $B_R = B_R(0)$.

As mentioned before, we use the notation $Q_{r_1,r_2}(t_0,x_0)$ to denote the following cylinder in space-time
\[ Q_{r_1,r_2}(t_0,x_0) := (t_0 - r_1, t_0] \times B_{r_2}(x_0). \]
When we write $Q_{r_1,r_2}$ we mean $Q_{r_1,r_2}(0,0)$. Note that we are specifying the scale in time and space separately. The two-parameter scaling of our equation (described in Section \ref{ss:scaling}) forces us to choose the time-scaling in terms of the spatial scale in a way that depends on the size of the gradient at the center.
$S^{d-1}$ denotes the unit sphere in $\R^d$. 

We write $C, c$ for positive constants that depend only on $d$, $s$ and $p$ unless otherwise specified. The values may change from line to line.
In particular, they may depend on $\sigma:=2-p(1-s) \in (0,2)$. No uniformity as $\sigma\to0$ or $\sigma\to2$ is asserted.  We sometimes write $C(\delta)$, $\eps(\delta)$, etc., for constants that depend additionally on the indicated parameter.

We use the notation $C(-1,0;L^{p-1}_{sp}(\R^d))$ to denote $C((-1,0];L^{p-1}_{sp}(\R^d))$, that is, continuity is assumed at the end point $t=0$.

\section{Preliminaries}\label{s:preliminaries}

Throughout the paper, the fractional $p$-Laplacian and all such integrals are understood in the principal value sense.

The linearization of $(-\Delta_p)^s$ at $u$ is the linear (in $v$) operator
\begin{equation} \label{eq:linearized}
    \LL_u v(t,x) := (p-1) \int_{\R^d} \frac{|u(t,x)-u(t,y)|^{p-2}}{|x-y|^{d+sp}} (v(t,x)-v(t,y))\dd y
    = \int_{\R^d} (v(t,x)-v(t,y))\, K_u(t;x,y) \dd y.
\end{equation}
The kernel $K_u(\cdot,\cdot;t)$ is symmetric in $x$ and $y$ at each fixed time and is given by
\begin{equation*} 
    K_u(t;x,y) := (p-1) \frac{|u(t,x)-u(t,y)|^{p-2}}{|x-y|^{d+sp}}.
\end{equation*}

When $u$ is differentiable in $x$, the kernel $K_u(t;\cdot,\cdot)$ is of homogeneity $-d-2+p(1-s)$ near the diagonal. Thus, $\LL_u$ is, at each time, a nonlocal operator of order $\sigma = 2-p(1-s) \in (0,2)$ for $p \in [2,2/(1-s))$. The condition $p<2/(1-s)$ ensures that the operator is generically of positive fractional order. 

The pointwise meaning of $\LL_u v(t,x)$ requires some care. As discussed in detail in \cite{GJS}, when $p > 1/(1-s)$, Lipschitz continuity of $u$ and $v$ in the spatial variable suffices for the integral in \eqref{eq:linearized} to be absolutely convergent. When $p \in [2, 1/(1-s)]$, the integral is no longer absolutely convergent at the diagonal in general, but $\LL_u v$ can be understood in the sense of an associated bilinear form, defined for $v$ in a suitable fractional Sobolev space. We use $\LL_u$ in this paper only through quadratic expressions of the form
\[
\int_{B} \LL_u v(t,x)\, \varphi(x) \dd x,
\]
for compactly supported test functions $\varphi$. This expression makes sense under mild regularity conditions, as analyzed in \cite[Section 2]{GJS}. To keep the exposition focused on the new ideas, we assume in the proofs below that $u$ is sufficiently regular in space for all our manipulations to be classical. The case of weak solutions is treated by an approximation argument, exactly as in \cite[Section 9]{GJS}.

\subsection{Weak and viscosity solutions}\label{s:weak/viscosity}

We introduce the notions of weak and viscosity solution to equation \eqref{eq:fracplap}.

\begin{definition}\label{def:weak-sol-definition}
We say that $u\colon(t_1,t_2] \times \mathbb{R}^d   \to \mathbb{R}$ is a weak subsolution to \eqref{eq:fracplap} in $(t_1,t_2] \times \Omega $ if 
\[
    u \in C\left(t_1,t_2;L^2_{\mathrm{loc}}(\Omega)\right) \cap L^p_{\mathrm{loc}}\left(t_1,t_2;W^{s,p}_{\mathrm{loc}}(\Omega) \right)\cap L_{\mathrm{loc}}^\infty(t_1,t_2;L_{sp}^{p-1}(\R^d))
\]
and for any compact set $\mathcal{K} \subset \Omega$ and any sub-interval $[\tau_1,\tau_2] \subset (t_1,t_2]$ 
\[
\begin{aligned}
&\int_{\mathcal{K}} u(t,x)\varphi(t,x)\,\dd x \bigg|_{\tau_1}^{\tau_2}
 - \int_{\tau_1}^{\tau_2}\int_{\mathcal{K}} u(t,x)\,\partial_t\varphi(t,x)\,\dd x \dd t \\
&\qquad
 + \frac12 \int_{\tau_1}^{\tau_2} \iint_{\R^d \times \R^d} \frac{|u(t,x) - u(t,y)|^{p-2}(u(t,x) - u(t,y))(\varphi(t,x) - \varphi(t,y))}{|x-y|^{d+sp}}\dd y\dd x \dd t
\le 0,
\end{aligned}
\]
for any nonnegative $\varphi$ such that
\[
   \varphi \in L^p(\tau_1,\tau_2;W^{s,p}_0(\mathcal{K})) \cap H^1(\tau_1,\tau_2;L^2(\mathcal{K})). 
   \]
   
A weak supersolution is defined similarly. Furthermore, a weak solution is both a weak subsolution and a supersolution.
\end{definition}

Now we define the notion of viscosity solution.

\begin{definition}\label{def:viscosity-solution}
We say that $u$ is a viscosity subsolution to \eqref{eq:fracplap} in $(t_1,t_2] \times \Omega$ if 
\[
    u\in \mathrm{USC}((t_1,t_2]\times\Omega)\cap C(t_1,t_2;L_{sp}^{p-1}(\R^d))
\]
and whenever there is a function $\phi\in C^2(Q_{r,r}(t_0,x_0))$ for some $Q_{r,r}(t_0,x_0) \subset (t_1,t_2] \times \Omega$, with $\phi$ touching $u$ from above at $(t_0,x_0)$ in $Q_{r,r}(t_0,x_0)$, we have 
\[
    \partial_t \phi(t_0,x_0) + (-\Delta_p)^s \phi_r (t_0,x_0)\leq 0,
\] 
where
\[
\phi_r(t,x)=\left\{\begin{array}{ll}
\phi(t,x) & \text{for}\; (t,x)\in Q_{r,r}(t_0,x_0),
\\[2mm]
u(t,x) & \text{otherwise}.
\end{array}
\right.
\]
A viscosity supersolution is defined similarly. Furthermore, a viscosity solution is both a viscosity subsolution and a supersolution.
\end{definition}

We shall now state some elementary estimates for the function $J_p(t)=|t|^{p-2}t$.

\begin{lemma}\label{l:appendix}
Suppose $p\geq2$ and $a, b\in \R$. Then 
\begin{align*}
&J_p(a) -J_p(b)= (p-1) \int_{0}^{1} |b+t(a-b)|^{p-2} (a-b)  \dd  t, \\
\intertext{ and}
&\frac{1}{C_p} (|b|+|a-b|)^{p-2}\leq \int_0^1 |b+t(a-b)|^{p-2} \dd t\leq C_p (|b|+|a-b|)^{p-2}.
\end{align*}
Moreover
\[
|J_p(a+b)-J_p(a)|
\leq
C_p\left(
|a|^{p-2}|b|+|b|^{p-1}
\right).
\]
Since $J_p$ is odd, this implies
\[
\begin{aligned}
&|J_p(a+b_+)+J_p(-a+b_-)|
\\
&\qquad\leq
C_p\left[
|a|^{p-2}\bigl(|b_+|+|b_-|\bigr)
+|b_+|^{p-1}
+|b_-|^{p-1}
\right],
\end{aligned}
\]
for some constant $C_p$, depending only on $p$. 
\end{lemma}

\subsection{Tail functions}\label{s:tails}

For $q>0$ and $\alpha>0$, the tail of a function $w$ is
\[
{\Tail}_{q,\alpha}(w;x,R) := \left( \int_{B_R^c(x)} \frac{|w(y)|^{q}}{|y-x|^{d+\alpha}} \dd y \right)^{\frac 1q},
\qquad {\Tail}_{q,\alpha}(w;R):={\Tail}_{q,\alpha}(w;0,R).
\]

By $L_{\alpha}^{q}(\R^d)$ we denote the $(q,\alpha)$--\textit{tail space} defined by
\[
L_{\alpha}^{q}(\R^d)\coloneqq \left\{ u \in L_{loc}^{q}(\R^d) : \int_{\R^d} \frac{|u(y)|^{q}}{(1+|y|)^{d+\alpha}  }\dd y < \infty \right\}.
\]

The following bounds are proved in \cite[Section 2]{GJS}.

\begin{lemma}\label{l:Tail_inequalities}
For $2< p<2/(1-s)$ and $R>0$,
\begin{gather*}
{\Tail}_{p-2,sp}(w;R)^{p-2}\leq C R^{\frac{p(1-s)-2}{p-1}}\, {\Tail}_{p-1,sp+1}(w;R)^{p-2},\\
{\Tail}_{p-1,sp+1}(w;R)^{p-1} \leq R^{-1}\,{\Tail}_{p-1,sp}(w;R)^{p-1},
\end{gather*}
with $C>0$ depending only on $d, \, s$ and $p$. More generally, for $p>2$ and any $0\leq\vartheta<\sigma/(p-1)$, we have
\[
{\Tail}_{p-2,\,sp-\vartheta}(w;R)^{p-2} \leq C(d,s,p,\vartheta)\, R^{\frac{\vartheta(p-1)-\sigma}{p-1}}\, {\Tail}_{p-1,sp+1}(w;R)^{p-2}.
\]
Moreover, tails centered at different points $x,x'$ with $|x-x'|\leq R/2$ are comparable up to a constant and a doubling of the radius.
\end{lemma}

\subsection{Scaling} \label{ss:scaling}~
The equation has a two-parameter scaling: one parameter fixes the spatial scale and the other fixes the size of the function. Given $\rho>0$ and $\lambda>0$, define
\[
u_{\rho,\lambda}(t,x) := \lambda^{-1} u(\lambda^{2-p}\rho^{sp}t,\rho x).
\]
The time scale is then forced by the homogeneity of the equation. Indeed,
\[
\partial_t u_{\rho,\lambda}(t,x)
= \lambda^{1-p}\rho^{sp}\, \partial_t u(\lambda^{2-p}\rho^{sp}t,\rho x),
\]
while
\[
(-\Delta_p)^s u_{\rho,\lambda}(t,x)
= \lambda^{1-p}\rho^{sp}\, (-\Delta_p)^s u(\lambda^{2-p}\rho^{sp}t,\rho x).
\]
Thus, if $u$ solves \eqref{eq:fracplap}, then ${u_{\rho,\lambda}}$ also solves \eqref{eq:fracplap} in the rescaled domain. The spatial gradient and tails scale as
\[
\nabla u_{\rho,\lambda}(t,x)
= \frac{\rho}{\lambda}\nabla u(\lambda^{2-p}\rho^{sp}t,\rho x),
\]
and, for any $q,\alpha>0$ and any $r>0$,
\begin{equation}\label{scale_tails}
{\Tail}_{q,\alpha}(u_{\rho,\lambda}(t,\cdot);r)^q
= \lambda^{-q}\rho^\alpha
{\Tail}_{q,\alpha}(u(\lambda^{2-p}\rho^{sp}t,\cdot); \rho r)^q.
\end{equation}
More generally, choosing $\lambda$ independently is useful when we want to normalize the size of $u$ or of $\nabla u$ while zooming to a spatial scale $\rho$.

In the following, it will be convenient to consider during the iteration the function
\begin{equation}\label{eq:spatial-centering}
w(t,x):=u(t,x)-u(t,0).
\end{equation}
In particular, note that a $L^\infty$ bound on $\nabla u$ implies a $L^\infty$ bound on $w$. Moreover $\nabla w=\nabla u$ and $K_w=K_u$. We use $w$ only in the spatial estimates: subtracting $u(t,0)$ adds the forcing $-\p_tu(t,0)$ to the equation. At a rescaled solution $u_n$, the same notation is $w_n(t,x)=u_n(t,x)-u_n(t,0)$.

\subsection{Smoothening of the solution}

As in the elliptic problem, we assume throughout the paper that the solution is smooth in $x$. Time derivatives are understood in the weak or viscosity sense, as appropriate. All estimates are independent of these regularity norms.

This smoothness assumption can be rigorously justified, as in \cite{GJS}, by perturbing the original equation and considering
\[
\p_t u_\varepsilon+(-\Delta_p)^su_\varepsilon-\varepsilon\Delta u_\varepsilon=0.
\]
By the argument in \cite[Lemma 9.4]{GJS}, we know that, for any $\varepsilon>0$, $u_\varepsilon$ is $C^{2,\alpha}$ in $x$ for every $\alpha< p(1-s)$ if $p \in [2,1/(1-s))$ and $u_\varepsilon$ is $C^{2,\alpha}$ in $x$ for all $\alpha \in(0,1)$ if $p \in [1/(1-s),2/(1-s))$. We can then follow the approximation procedure in \cite[Section 9]{GJS} to make the argument fully rigorous.

\section{The linearized equation}\label{s:linearized}

In this section we localize the directional derivatives of $u$ and compute the equation that they satisfy in terms of the linearized operator $\LL_u$ defined in \eqref{eq:linearized}. Throughout this section, we will always assume $u$ is a weak solution of \eqref{eq:fracplap} which is smooth in $x$. 

Let $\eta : \R^d \to [0,1]$ be a smooth radially symmetric cutoff function such that $\eta(x) = 1$ for $|x| \leq 3/2$ and $\eta(x) = 0$ for $|x| \geq 7/4$. For any unit vector $e \in S^{d-1}$ and any $R>0$, we define
\begin{equation} \label{eq:def_ve}
     v_e(t,x) := \eta(x/R) \, (e \cdot \nabla u(t,x)).
\end{equation}
The function $v_e$ coincides with the directional derivative $e \cdot \nabla u$ in $\R \times B_R$ (at every fixed time). It is globally bounded in space, with compact spatial support contained in $B_{2R}$. The dependence on $R$ is suppressed in the notation.

The equation satisfied by $v_e$ in $B_R$ is given by the next lemma. The result is the parabolic version of \cite[Lemma 3.3]{GJS}.

\begin{lemma}[The equation for localized directional derivatives]\label{l:eqforv_e_scaled}
Let $T,R>0$, and let $u$ be a solution of \eqref{eq:fracplap} in $Q_{T,2R}$. Then $v_e$ satisfies, for $(t,x)\in Q_{T,R}$,
\begin{equation} \label{eq:linearized_est}
       |\p_t v_e(t,x) + \LL_{u} v_e(t,x)| \leq  C \left( R^{-1-sp}\, \| w\|^{p-1}_{L^{\infty}(Q_{T,2R})} +   {\Tail}_{p-1,sp+1}(w(t,\cdot);R)^{p-1}\right),
\end{equation}
where $C$ depends only on $d$, $s$, and $p$.
\end{lemma}

The proof of this result follows the lines of the elliptic argument. The idea is to split the domain of integration of $(-\Delta_p)^{s}u$ using $\eta(x/R)$ and then differentiate along the direction $e$. A remainder term is obtained by integrating by parts and differentiating the cutoff $\eta$ and the kernel $|x-y|^{-d-sp}$. In the same way as in \cite[Lemma 3.3]{GJS}, one obtains \eqref{eq:linearized_est}.

Indeed, the localization gives, in $B_R$,
\[
\p_t v_e+\LL_u v_e
=\int_{\R^d}J_p(u(t,x)-u(t,y))\,
e\cdot\nabla_y\bigl((1-\eta(y/R))|x-y|^{-d-sp}\bigr)\dd y.
\]
Only spatial increments occur on the right, so the bounds use $w$ without changing the equation for $v_e$.

\section{Reducing the norm of the gradient}\label{s:reduce_norm}

This section contains the most technical part of the paper. We prove the parabolic counterpart of \cite[Lemma~4.1]{GJS}. The lemma is a type of improvement-of-oscillation result for the directional derivatives of $u$, applied through the localized function $v_e$. As in \cite{GJS}, the proof is a modified De Giorgi iteration, where the level sets are estimated through their measures (rather than through the $L^2$ norm of truncations) and there is no global coercivity estimate. Throughout this section, we will always assume $u$ is a weak solution of \eqref{eq:fracplap} which is smooth in $x$.

The lemma asserts that,  if $\nabla u$ deviates from a fixed unit vector $e$ on a set of positive space-time measure (in the past), then the directional derivative $e \cdot \nabla u$ has a quantifiable improvement in the future.

 We first state the main result, which is the parabolic version of the elliptic Lemma~4.1 of \cite{GJS}.

\begin{lemma}[Improvement of oscillation] \label{l:iteration}
For any $r,\mu,\tau \in (0,1/2)$, there exist $\delta_1, \delta_2, \eps_1 \in (0, 1/2)$, and $K_0 \in \mathbb{N}$, depending only on $r,\,\mu,\,\tau, \, d,\, s,$ and $p$, such that the following statement holds.

Assume that, for some unit vector $e \in S^{d-1}$, the function $u$ satisfies:
    \begin{enumerate}
    \item $u$ is a solution of \eqref{eq:fracplap} in $Q_{1,2^{K_0+1}}$,
    \item $\|\nabla u\|_{L^{\infty}(Q_{1,2^{n}})} \leq (1-\delta_1)^{-n}$ for $n=0,1,\dots,K_0$,
    \item $\sup_{t \in (-1,0]} {\Tail}_{p-1,sp+1}(u(t,\cdot)-u(t,0);2^{K_0-1}) \leq \eps_1$,
    \item $|\{ (t,x) \in (-1,-2\tau) \times B_1 : \nabla u(t,x) \notin B_r(e) \}| \geq \mu$.
    \end{enumerate}
Then
\[
e \cdot \nabla u(t,x) \leq 1-\delta_2 \quad \text{ for } (t,x) \in (-\tau/16,0] \times B_{1/2}.
\]
\end{lemma}

\begin{remark}\label{r:statement-comments}
Notice that Assumption (2) for $n=0$ implies $\|w\|_{L^\infty(Q_{1,1})} = \|u(t,\cdot) -u(t,0)\|_{L^\infty(Q_{1,1})}\leq 1$.

We use two separate small parameters: $\delta_1$ quantifies the dyadic growth of the gradient in assumption (2), while $\delta_2$ quantifies the improvement in the conclusion. They play different roles in the proof and, crucially, $\delta_1$ will be chosen small depending on a quantity that determines $\delta_2$ but does not depend on $\delta_1$. The lemma will later be applied with the single parameter $\delta := \min(\delta_1,\delta_2)$, for which both assumption (2) and the conclusion hold. In fact the result stays true if we decrease $\delta_1$, $\delta_2$ or $\eps_1$, or increase $K_0$.
\end{remark}
For a unit vector $e$ and $r>0$, define
\begin{equation*} 
    \A_r(t) := \{ x \in B_1 : \nabla u(t,x) \notin B_r(e) \}.
\end{equation*}

Assumption (4) is equivalent to $\int_{-1}^{-2\tau}|\A_r(t)|\dd t \geq \mu$. The first step is to select a region in space from which the set $\A_r(t)$ is seen through a fixed cone of directions. For $x\in\R^d$ let
\begin{equation}\label{def:cone}
C(x) := \{ x + \rho \nu \ :\ \rho \in (0,2),\ \nu \in S^{d-1},\ |\nu \cdot e| > 1/6 \}
\end{equation}
be the cone with vertex $x$, axis $e$, opening $1/6$ and radius $2$. 

The following lemma is in the same spirit of \cite[Lemma 4.4]{GJS}. We omit the proof.

\begin{lemma}[Choice of the center] \label{l:pickaball}
Assume $\int_{-1}^{-2\tau} |\A_r(t)| \dd t \geq \mu$. There is a point $x_1 \in \partial B_{1/2}$, with $x_1 = e/2$ or $x_1=-e/2$, such that, defining
\begin{equation} \label{def:hatg}
    g(t) := |\A_r(t) \cap C^*(x_1)|,
    \qquad C^*(x_1) := \bigcap_{x\in B_{1/8}(x_1)} C(x),
\end{equation}
we have
\begin{equation*} 
\int_{-1}^{-2\tau} g(t) \dd t \geq \mu/2 .
\end{equation*}
In particular $|\A_r(t) \cap C(x)| \geq g(t)$ for every $x \in B_{1/8}(x_1)$ and every $t$.
\end{lemma}

The core of the section is the following localized statement. Since we will need to re-apply it on later time windows during the proof of Lemma~\ref{l:iteration}, we state it with the time window and the center as parameters.

\begin{lemma}[The De Giorgi engine] \label{l:DGengine}
Let $r,\mu \in (0,1/2)$, and $-1 \leq t_a < t_b < 0$. There exist $\delta_1,\delta_2,\eps_1 \in (0,1/2)$ and $K_0 \in \mathbb N$, depending only on $r$, $\mu$, $|t_b|$, $d$, $s$, and $p$ such that the following holds.

Assume $u$ satisfies assumptions (1)--(3) of Lemma~\ref{l:iteration} and, for some $x_1 \in \overline B_{1/2}$ and some unit vector $e$,
\begin{equation}\label{eq:engine-reservoir}
\int_{t_a}^{t_b} |\A_r(t) \cap C^*(x_1)| \dd t \geq \mu,
\qquad\text{where } C^*(x_1) := \bigcap_{x\in B_{1/8}(x_1)} C(x).
\end{equation}
Then
\[
e\cdot \nabla u(t,x) \leq 1-\delta_2 \quad \text{ for } (t,x) \in (t_b/2,0] \times B_{1/16}(x_1).
\]
\end{lemma}

Like in the elliptic case, we prove Lemma \ref{l:iteration} by applying Lemma \ref{l:DGengine} more than once. We first apply Lemma \ref{l:DGengine} with $x_1$ as the one given by Lemma \ref{l:pickaball}. Then, we use the deduced upper bound for $(e \cdot \nabla u)$ around $x_1$, at a later time, to apply Lemma \ref{l:DGengine} again and enlarge the radius of the ball where we have the lower bound. Hence we focus on the proof of Lemma \ref{l:DGengine}.

From now on, we place ourselves in the setting of Lemma~\ref{l:DGengine}: the center $x_1\in \overline B_{1/2}$, the  time interval $(t_a,t_b)$ and the mass $\mu$ are fixed, assumptions (1)--(3) of Lemma~\ref{l:iteration} and \eqref{eq:engine-reservoir} hold, and we write
\begin{equation}\label{def:hatg-engine}
g(t) := |\A_r(t) \cap C^*(x_1)| , \qquad \int_{t_a}^{t_b}g(t)\dd t \geq \mu .
\end{equation}
All constants below may depend on $d,s,p$ and on $|t_b| \in (0,1]$ without further mention. We may and do assume $r\leq 1/15$ throughout. We keep a record of the choices and dependencies of constants in Remark~\ref{r:constants-order}.

As in Lemma \ref{l:eqforv_e_scaled}, we define $v_e$ to be a properly localized version of the directional derivative of $u$ with $R=2^{K_0-1}$,
\begin{equation} \label{eq:def_ve_dyadic}
v_e(t,x) := \eta(2^{1-K_0}x)\, (e\cdot \nabla u(t,x)),
\end{equation}
where $\eta$ is the fixed cutoff appearing in \eqref{eq:def_ve}. Thus $v_e = e\cdot\nabla u$ in $B_{2^{K_0-1}}$ and $\supp v_e(t,\cdot) \subset B_{2^{K_0}}$. Assumption (2) of Lemma~\ref{l:iteration} gives
\begin{equation}\label{eq:ve-growth}
|v_e(t,x)| \leq (1-\delta_1)^{-n} \ \text{ in } Q_{1,2^n} \ (0\leq n\leq K_0),
\qquad
|v_e(t,x)| \leq (2|x|)^{\alpha_0} \ \text{ for } (t,x) \in (-1,0]\times B_1^c,
\end{equation}
where $2^{-\alpha_0} = 1-\delta_1$. In particular $|v_e|\leq 1$ in $(-1,0]\times B_1$.

Condition (2) gives, setting $w(t,x)=u(t,x)-u(t,0)$, for $k = 0,1,\dots,K_0$,
\begin{equation}\label{eq:dyadic_gradient_w}
\|w\|_{L^\infty(Q_{1,2^{k}})} \leq \|w\|_{L^\infty(Q_{1,1})} + 2^{k}\,\|\nabla u\|_{L^\infty(Q_{1,2^{k}})} \leq  1+ 2^{k}(1-\delta_1)^{-k}.
\end{equation}

\begin{lemma} \label{l:ve-eq-small-rhs}
Let $u$ be a function satisfying the assumptions of Lemma \ref{l:iteration}. Let $v_e$ be defined as in \eqref{eq:def_ve_dyadic}. Then, it solves the equation
\begin{equation} \label{eq:equation_ve}
    |\p_t v_e(t,x) +\LL_u v_e(t,x)| \leq \eps_0 \quad \text{ for } (t,x)\in Q_{1,1},
\end{equation}
where $\eps_0$ is arbitrarily small if we choose $\eps_1$ and $\delta_1$ small and $K_0$ large in Lemma \ref{l:iteration}.
\end{lemma}

\begin{proof}
We apply Lemma~\ref{l:eqforv_e_scaled} with $R=2^{K_0-1}$ and use Assumptions \textit{(2)}--\textit{(3)} to obtain for each $(t,x) \in Q_{1,1}$ that
\begin{equation}\label{def:eps_0}
\begin{aligned}
|\p_t v_e + \LL_u v_e| &\leq C \left( 2^{-K_0(1+sp)} \|w\|^{p-1}_{L^\infty(Q_{1,2^{K_0}})} + \sup_{t \in (-1,0]}{\Tail}_{p-1,sp+1}(w(t,\cdot);2^{K_0-1})^{p-1} \right)  \\
 & \leq C \left( 2^{-K_0(1+sp)}  + \left((1-\delta_1)^{-(p-1)} 2^{-2+p(1-s)} \right)^{K_0}+ \eps_1  \right) =: \eps_0. 
\end{aligned} 
\end{equation}
We point out that $\eps_0$ can be made arbitrarily small by taking $K_0$ large enough, since $(1-\delta_1)^{-(p-1)} 2^{p(1-s)-2}  <1$ for $\delta_1$ small and $\eps_1$ is small.    
\end{proof}

\subsection{The De Giorgi setup}\label{s:DGsetup}

We now outline the De Giorgi iteration in the parabolic setting. Let us recall the constants $r,\mu>0$ from Lemma~\ref{l:iteration}, the point $x_1$ provided by Lemma~\ref{l:pickaball}, and the function $g$ defined in \eqref{def:hatg}. The iterative process is divided into two distinct phases. The first is a finite preliminary phase consisting of $N \geq 2$ iterations (where the integer $N$ will be chosen later), during which the height decreases geometrically while the spatial scale remains fixed. The role of this preliminary phase is to ensure the smallness assumption to start the recurrence law is satisfied.
The second is the nonlinear phase, where we employ the standard shrinking spatial cut-offs and execute the nonlinear De Giorgi recursion.

\medskip
 
\noindent\textbf{Bump functions in space.}

Fix the radius $r_1 := \tfrac{5}{64}$ and let
$\bar\eps \in (0,1/8]$ be a small parameter which will be fixed in Proposition \ref{p:mass-transfer}  (it will depend only on $d,s,p$, and $|t_b|$). 
Note $(1+\bar \eps)r_1 \leq \tfrac98 r_1 < \tfrac18$, so that $B_{(1+\bar \eps)r_1}(x_1) \subset B_{1/8}(x_1) \subset B_{3/4}$ since $x_1\in \overline B_{1/2}$.
Let $\eta_0$ be a smooth cutoff satisfying $\eta_0 = 1$ on $\overline B_{r_1}(x_1)$, $\supp\eta_0 = \overline B_{(1+\bar \eps)r_1}(x_1)$, and $\|\eta_0\|_{Lip} \leq C/\bar\eps $.

The cut-off $\eta_0$ serves as the fixed spatial cutoff during the preliminary phase. After $N$ iterations, we transition to the nonlinear phase, introducing a family of bump functions with shrinking spatial supports.
To this end, we define a sequence of cutoff functions $\eta_k$ with $\eta_k := \eta_0$ for $1 \leq k \leq N$.

For $k \geq N + 1$, we define $\eta_k$ as a smooth radial bump function satisfying the following properties:
\begin{itemize} 
\item $\eta_k = 1$ on $B_{1/16 + 4^{N-k-3}}(x_1)$;
\item $\eta_k = 0$ outside $B_{1/16 + 4^{N-k-2}}(x_1)$;
\item $|\nabla \eta_k| \leq C2^{2(k-N)}$ and $|\nabla^2 \eta_k| \leq C2^{4(k-N)}$.
\end{itemize}

Thus the first $N$ levels use the same spatial scale, and the shrinking of the spatial supports starts only at level $N+1$.

\noindent\textbf{The time-dependent heights.} We define $a_1(t)$ as the solution, in the a.e. sense, of the ODE
\begin{equation} \label{eq:ODEa}
\begin{cases}
a_1'(t) = \kappa\, g(t)\, \one_{(t_a,t_b)}(t) , \\
a_1(-1) = 0,
\end{cases}
\end{equation}
where $\kappa>0$ is fixed in \eqref{eq:kappa-heights-cap} below. The forcing is proportional to the visible mass $g(t)$. The solution is explicit
\begin{equation*}
a_1(t) = \kappa \int_{\min(t,t_a)}^{\min(t,t_b)} g(\rho) \dd\rho \geq 0 ,
\end{equation*}
so $a_1$ is Lipschitz and \emph{nondecreasing}, $a_1 \equiv 0$ on $[-1,t_a]$, and $a_1$ is constant on $[t_b,0]$. We record two elementary consequences
\begin{align}
& a_1(t) \leq \bar a := \kappa\, |B_1| \qquad \text{for all } t,\nonumber \\
& a_1(t) = \kappa \int_{t_a}^{t_b} g(\rho) \dd\rho \geq \kappa\, \mu \qquad \text{for } t \in [t_b, 0], \label{eq:a1-lower}
\end{align}
using $g \leq |B_1|$ and $t_b - t_a \leq 1$ for the first, and \eqref{def:hatg-engine} for the second.

For the preliminary phase, with $\lambda\in(0,1/4]$ the contraction ratio fixed in \eqref{def:lambda} below, define
\begin{equation*}
    a_k(t):=\lambda^{k-1}\,a_1(t),\qquad 1\leq k\leq N,
\end{equation*}
and set
\begin{equation}\label{def:delta-from-aN}
    \delta_0:= \tfrac12\, \lambda^{N-1}\, \kappa\, \mu ,
\end{equation}
so that, by \eqref{eq:a1-lower},
\begin{equation*}
    a_N(t)\geq 2\delta_0\qquad\text{for }t\in[t_b,0].
\end{equation*}
The number $\delta_0$ depends only on $r,\mu,|t_b|,d,s,p$ (through $\kappa$, $\lambda$, $N$), and not on the solution $u$; it determines the final gain, $\delta_2 = \delta_0/2$. Note also that $\delta_0 \leq \tfrac12\kappa\mu \leq \bar a/2$.

For the nonlinear phase, define the times
\[
    t_j := \frac{t_b}{2}\big(1+2^{-j+1}\big), \qquad j\geq1,
\]
so that $t_1 = t_b$, $t_j \nearrow t_b/2$, and $t_{j+1}-t_j = |t_b|\,2^{-j-1}$. With $\theta_j := 1+2^{-j} \in (1,\tfrac32]$, set
\begin{equation}\label{def:nonlinear-heights}
    a_{N+j}(t) := \begin{cases}
    0 & \text{ for } t \in [-1,t_j], \\[2pt]
    \theta_j\, \delta_0\, \dfrac{t-t_j}{t_{j+1}-t_j} & \text{ for } t \in [t_j,t_{j+1}], \\[4pt]
    \theta_j\, \delta_0 & \text{ for } t \in [t_{j+1},0].
    \end{cases}
\end{equation}
Thus every nonlinear height vanishes before $t_b$ and sits at its plateau $\theta_j\delta_0$ from time $t_{j+1} < t_b/2$ on. Since $\delta_0 \leq \bar a/2$, all heights, in both phases, are bounded by $\bar a$. See Figure \ref{fig:two-phase-heights} for a direct comparison between the different heights in both phases.

In the next lemma we collect some elementary properties of the heights. The proof is straightforward and omitted.

\begin{lemma} \label{l:a-bounds}
The heights defined above satisfy the following properties:
\begin{enumerate}
\item $a_{k+1}(t) \leq a_k(t)$ for every $t \in [-1,0]$ and every $k\geq1$;
\item $0\leq a_{N+j}(t)\leq \tfrac32\delta_0$ and $(a_{N+j}')_+(t)\, a_{N+j}(t) \leq C\,2^{j}\,\delta_0^2/|t_b|$ for every $j\geq1$ and a.e.\ $t$;
\item $a_{N+j}(t)\geq\delta_0$ for every $j\geq1$ and every $t\in[t_b/2,0]$.
\end{enumerate}
\end{lemma}

\noindent\textbf{Barriers, level sets, truncations.} Define
\begin{equation*} 
\varphi_k(t,x) := 1 - a_k(t)\, \eta_k(x) + \eps_2\, (t+1), \qquad k \geq 1,
\end{equation*}
where $\eps_2 := \eps_0 + \eps(\delta_1)$, with $\eps_0$ from Lemma~\ref{l:ve-eq-small-rhs} and $\eps(\delta_1)$ from Lemma~\ref{l:bound_J_Tail} below. The constant $\delta_0$ of \eqref{def:delta-from-aN} does not depend on $\delta_1,\eps_1,K_0$; we may therefore first fix $\delta_0$ and then choose $\delta_1$ small and $K_0$ large so that
\begin{equation}\label{eq:eps2-small}
    \eps_2 \leq \frac{\delta_0}{2} .
\end{equation}
Define the level sets and truncations
\begin{equation} \label{def_levelsetsAk}
A_k(t):= \left\{ x \in B_1 : v_e(t,x) > \varphi_k(t,x) \right\},
\qquad
v_k(t,x):= \begin{cases}
 \big( v_e - \varphi_k \big)_+(t,x) & \text{ for } x \in B_1, \\
0 & \text{ for } x \notin B_1.
\end{cases}
\end{equation}

Since $|v_e|\leq 1$ in $B_1$ and $\eps_2(t+1)\geq0$, $x \in A_k(t)$ forces $a_k(t)\eta_k(x) > \eps_2(t+1) \geq 0$. In particular
\begin{equation}\label{eq:vk-pointwise}
A_k(t) \subset \supp\eta_k \subset \overline B_{(1+\bar \eps)r_1}(x_1), \qquad 0\leq v_k(t,x) \leq a_k(t)\,\eta_k(x),
\qquad A_{k+1}(t)\subset A_k(t),
\end{equation}
Moreover $A_k(t) = \emptyset$ whenever $a_k(t)=0$; in particular $v_k(-1,\cdot)=0$ and, for the nonlinear levels, $A_{N+j}(t)=\emptyset$ for $t\leq t_j$. We also know that, for $j\geq 1$,
\begin{align}\label{eq:prop-AN}
     A_{N+j+1}(t)\neq \emptyset \quad \text{ implies } \quad a_{N+j+1}(t)>0\, \text{ and } a_{N+j}(t)=\theta_j \delta_0>\delta_0.
\end{align}

Finally, since all heights are bounded by $\bar a$, $x \in A_k(t)$ implies $e\cdot\nabla u(t,x) = v_e(t,x) > 1-\bar a$, and $|\nabla u(t,x)|\leq1$ in $B_1$. Thus, if $x \in A_k(t)$,
\begin{equation}\label{eq:alignment}
|\nabla u(t,x) - e| \leq 2\sqrt{\bar a} .
\end{equation}

The following lemma collects some elementary bounds that will be used in the nonlinear phase. The proof is a straightforward application of Chebyshev's inequality.

\begin{lemma}[Chebyshev bounds, nonlinear phase] \label{l:chebi}
For every $t \in (-1,0)$ and every $j \geq 1$:
\begin{align*}
c2^{-j} a_{N+j}(t)\, |A_{N+j+1}(t)|
\leq \int_{B_1} v_{N+j}(t,x)\dd x
&\leq a_{N+j}(t)\, |A_{N+j}(t)|,
\\
c2^{-2j} a_{N+j}(t)^2 |A_{N+j+1}(t)|
&\leq \int_{B_1} v_{N+j}(t,x)^2 \dd x  .
\end{align*}
\end{lemma}

%%%%%%% Begin Tikz

\begin{figure}[ht]
\centering
\begin{tikzpicture}[
    x=0.82cm,
    y=0.88cm,
    >=stealth,
    line cap=round,
    line join=round
]

% -------------------------------------------------------------------------
% Styles
% -------------------------------------------------------------------------
\tikzset{
    axis/.style={black, thick},
    guide/.style={dashed, line width=0.5pt},
    timeguide/.style={black!40, dashed, line width=0.5pt},
    curve/.style={line width=1.5pt},
    levelnode/.style={anchor=east, font=\small}
}

% -------------------------------------------------------------------------
% Time coordinates
% -------------------------------------------------------------------------
\coordinate (xminusone) at (0.8,0);   % -1
\coordinate (xta)       at (1.4,0);   % t_a
\coordinate (xtone)     at (4.5,0);   % t_1=t_b=-2 tau
\coordinate (xttwo)     at (6.0,0);   % t_2
\coordinate (xtthree)   at (7.3,0);   % t_3
\coordinate (xtfour)    at (8.3,0);   % t_4
\coordinate (xminustau) at (9.2,0);   % -tau=t_b/2
\coordinate (xzero)     at (12.3,0);  % 0
\coordinate (xright)    at (13.0,0);

% -------------------------------------------------------------------------
% Schematic heights (the vertical spacing is not to scale)
% -------------------------------------------------------------------------
\def\yaone{7.15}
\def\yatwo{5.55}
\def\yaN{3.75}
\def\ythetaone{2.75}
\def\ythetatwo{2.15}
\def\ythetathree{1.78}
\def\ythetainfty{1.45}

% -------------------------------------------------------------------------
% Axes and time labels
% -------------------------------------------------------------------------
\draw[axis] (0.65,0) -- (xright);
\draw[axis] (0.65,0) -- (0.65,7.65);

\node[below] at (xminusone) {$-1$};
\node[below] at (xta) {$t_a$};
\node[below] at (xtone) {$t_1=t_b=-2\tau$};
\node[below, red!90!black] at (xttwo) {$t_2$};
\node[below, red!70!black] at (xtthree) {$t_3$};
\node[below, red!50!black] at (xtfour) {$t_4$};
\node[below] at (xminustau) {$-\tau$};
\node[below] at (xzero) {$0$};

% -------------------------------------------------------------------------
% Horizontal guide levels and their values
% -------------------------------------------------------------------------
\draw[guide, green!80!black] (0.70,\yaone) -- (12.15,\yaone);
\draw[guide, green!65!black] (0.70,\yatwo) -- (12.15,\yatwo);
\draw[guide, green!45!black] (0.70,\yaN) -- (12.15,\yaN);
\draw[guide, red!85!black]   (0.70,\ythetaone) -- (12.15,\ythetaone);
\draw[guide, red!65!black]   (0.70,\ythetatwo) -- (12.15,\ythetatwo);
\draw[guide, red!45!black]   (0.70,\ythetathree) -- (12.15,\ythetathree);
\draw[guide, black!65]       (0.70,\ythetainfty) -- (12.15,\ythetainfty);

\node[levelnode, green!80!black] at (0.55,\yaone)
    {$a_1(t_b)=\kappa\displaystyle\int_{t_a}^{t_b}g(\rho)\,\dd\rho$};
\node[levelnode, green!65!black] at (0.55,\yatwo)
    {$a_2(t_b)=\lambda a_1(t_b)$};
\node[levelnode, green!45!black] at (0.55,\yaN)
    {$a_N(t_b)=\lambda^{N-1}a_1(t_b)\geq 2\delta_0$};
\node[levelnode, red!85!black] at (0.55,\ythetaone)
    {$\theta_1\delta_0$};
\node[levelnode, red!65!black] at (0.55,\ythetatwo)
    {$\theta_2\delta_0$};
\node[levelnode, red!45!black] at (0.55,\ythetathree)
    {$\theta_3\delta_0$};
\node[levelnode, black] at (0.55,\ythetainfty) {$\delta_0$};

% -------------------------------------------------------------------------
% Vertical guide times
% -------------------------------------------------------------------------
\draw[timeguide] (xta) -- (1.4,7.45);
\draw[timeguide] (xtone) -- (4.5,7.45);
\draw[timeguide, red!50!black] (xttwo) -- (6.0,\ythetaone);
\draw[timeguide, red!70!black] (xtthree) -- (7.3,\ythetatwo);
\draw[timeguide, red!90!black] (xtfour) -- (8.3,\ythetathree);
\draw[timeguide] (xminustau) -- (9.2,\ythetainfty);

% -------------------------------------------------------------------------
% Preliminary heights: a_2 and a_N are vertical rescalings of a_1
% -------------------------------------------------------------------------
\draw[curve, green!80!black]
    (xminusone) -- (xta)
    -- (1.95,{0.20*\yaone})
    -- (2.35,{0.20*\yaone})
    -- (2.90,{0.47*\yaone})
    -- (3.25,{0.47*\yaone})
    -- (3.80,{0.76*\yaone})
    -- (4.05,{0.76*\yaone})
    -- (4.50,\yaone)
    -- (12.30,\yaone);

\draw[curve, green!65!black]
    (xminusone) -- (xta)
    -- (1.95,{0.20*\yatwo})
    -- (2.35,{0.20*\yatwo})
    -- (2.90,{0.47*\yatwo})
    -- (3.25,{0.47*\yatwo})
    -- (3.80,{0.76*\yatwo})
    -- (4.05,{0.76*\yatwo})
    -- (4.50,\yatwo)
    -- (12.30,\yatwo);

\draw[curve, green!45!black]
    (xminusone) -- (xta)
    -- (1.95,{0.20*\yaN})
    -- (2.35,{0.20*\yaN})
    -- (2.90,{0.47*\yaN})
    -- (3.25,{0.47*\yaN})
    -- (3.80,{0.76*\yaN})
    -- (4.05,{0.76*\yaN})
    -- (4.50,\yaN)
    -- (12.30,\yaN);

\node[green!80!black, right] at (12.45,\yaone) {$a_1$};
\node[green!65!black, right] at (12.45,\yatwo) {$a_2=\lambda a_1$};
\node[green!45!black, right] at (12.45,\yaN) {$a_N=\lambda^{N-1}a_1$};

% -------------------------------------------------------------------------
% Nonlinear heights
% -------------------------------------------------------------------------
\draw[curve, red!85!black]
    (xminusone) -- (xtone)
    -- (6.00,\ythetaone)
    -- (12.30,\ythetaone);

\draw[curve, red!65!black]
    (xminusone) -- (xttwo)
    -- (7.30,\ythetatwo)
    -- (12.30,\ythetatwo);

\draw[curve, red!45!black]
    (xminusone) -- (xtthree)
    -- (8.30,\ythetathree)
    -- (12.30,\ythetathree);

\node[red!85!black, right] at (12.45,\ythetaone) {$a_{N+1}$};
\node[red!65!black, right] at (12.45,\ythetatwo) {$a_{N+2}$};
\node[red!45!black, right] at (12.45,\ythetathree) {$a_{N+3}$};

% -------------------------------------------------------------------------
% Limiting height
% -------------------------------------------------------------------------
\draw[curve, black]
    (xminusone) -- (xminustau)
    -- (9.20,\ythetainfty)
    -- (12.30,\ythetainfty);

\node[black, right] at (12.45,\ythetainfty) {$a_\infty$};

\end{tikzpicture}

\caption{
Schematic time-dependent heights $a_k$, both in the preliminary phase in green, and in the nonlinear phase in red. Figure not up to scale.
}
\label{fig:two-phase-heights}
\end{figure}
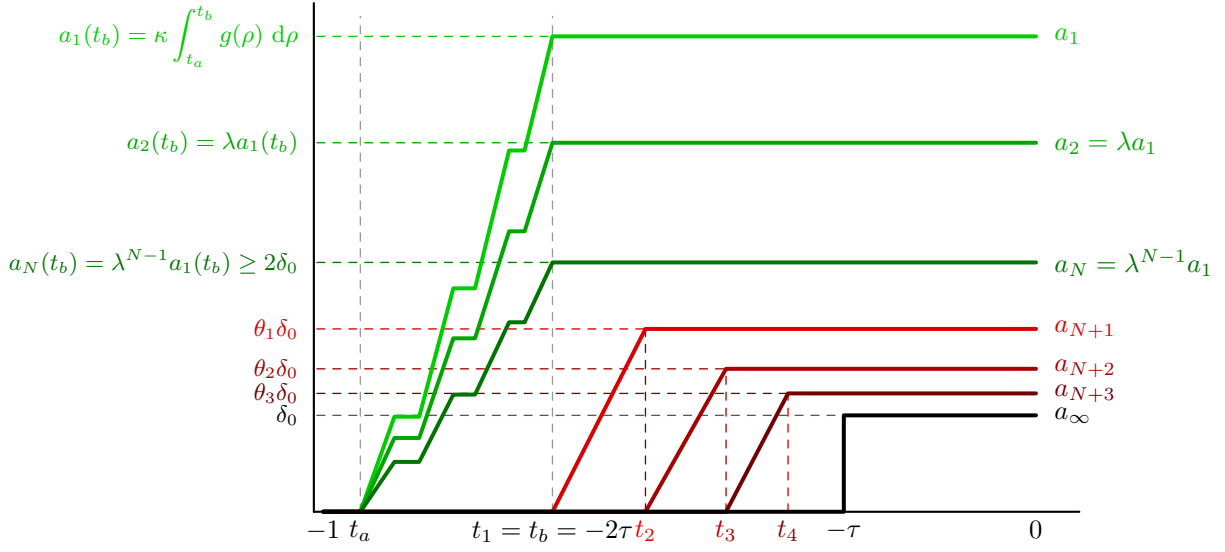

%%%%%%% End Tikz

\begin{remark}(Order of the choice of constants)\label{r:constants-order}

For the reader's convenience we list the order in which the constants are fixed; each depends only on the previous ones (besides $d,s,p$, as well as on $|t_b|$, and $r,\mu,\tau$, on which everything is allowed to depend). This ordering is what makes the scheme consistent.
\begin{enumerate}
\item The radius $r_1$, the cone opening $1/6$, the constant $c_0=1/42$ of Lemma~\ref{l:elementary-1D}, the cone-coercivity constant $c_A$ of Lemma~\ref{l:cone-coercivity} and $\mu_0$ of Proposition~\ref{p:mass-transfer} are all universal constants.

\item The tail constant $C_T$ and the function $\eps(\cdot)$ of Lemma~\ref{l:tails_of_kernel}, as well as the constant $C_0$ in \eqref{eq:master-preliminary}, depend only on $d,s,p$. These bounds are uniform for $\delta_1,\eps_1$ sufficiently small.

\item The nonlinear threshold $\omega_\star$ appearing in Proposition~\ref{p:linear-preliminary-contraction} depends only on $d,s,p,|t_b|$ and is fixed in the proof of Lemma~\ref{l:DGengine}. The constants $\bar\eps,\lambda,\gamma_\star$ and $N$ below therefore depend only on $d,s,p,|t_b|$. Notice that this threshold can be taken even smaller if needed. For the proof of Proposition~\ref{p:mass-transfer} we need $\omega_\star\leq\mu_0/2$.

\item The annulus width $\bar\eps$ in \eqref{eq:pick_eps} depends on $\omega_\star$ and is fixed in Lemma~\ref{l:cube}.

\item The contraction ratio $\lambda$ in \eqref{def:lambda} depends on $C_0,c_A,\mu_0,\omega_\star,\bar\eps$.

\item The mass-transfer constant $\gamma_\star$ in Proposition~\ref{p:mass-transfer} depends on $\omega_\star$ and $C_0$; then the number of preliminary levels $N$ is fixed in the proof of Proposition~\ref{p:linear-preliminary-contraction} depending on $\gamma_\star$.

\item The forcing size $\kappa$ in \eqref{eq:kappa-heights-cap} depends on $r$ and $c_A$.

\item The height gain $\delta_0$ in \eqref{def:delta-from-aN} and $\delta_2:=\delta_0/2$ also depend on $r$ and $\mu$.

\item Finally, we choose $\delta_1$ small enough that
\[
(1-\delta_1)^{-(p-1)}2^{p(1-s)-2}<1,
\qquad
\eps(\delta_1)\leq\frac{\delta_0}{4}.
\]
Then we choose $\varepsilon_1$ small and $K_0$ large enough that $\eps_0\leq\delta_0/4$, with $\eps_0$ defined in \eqref{def:eps_0}. Thus $\eps_2=\eps_0+\eps(\delta_1)\leq\delta_0/2$.
\end{enumerate}
\end{remark}

\subsection{The energy inequality}\label{s:energy}

Testing \eqref{eq:equation_ve} against $v_k$ we end up with the following energy inequality. The argument to obtain the $J$ terms is identical to the one in \cite{GJS} so we skip it.

\begin{lemma} \label{l:energy_inequality}
For almost every $t \in (-1,0]$ and every $k \geq 1$,
\begin{equation} \label{eq:main_ineq}
\begin{aligned}
    \frac{1}{2}\frac{\dd}{\dd t} \int_{B_1} v_k(t,x)^2\dd x  & +  J_{Tail,k}(t)
+ J_{\varphi,k}(t) + J_{gap,k}(t) + J_{r,k}(t) \\ &- a_k'(t) \int_{B_1} v_k\, \eta_k\, \dd x + \eps_2 \int_{B_1} v_k \dd x \leq \eps_0 \int_{B_1} v_k \dd x,
\end{aligned}
\end{equation}
where
\begin{equation} \label{eq:defJgapk}
\begin{aligned}
J_{Tail,k}(t) &= \int_{\R^d \setminus B_1}\int_{B_1} (\varphi_k(t,x)-v_e(t,y))\,v_k(t,x)\, K_u(t;x,y)\,\dd x\dd y, \\
J_{\varphi,k}(t) &= \iint_{B_1\times B_1}(\varphi_k(t,x)-\varphi_k(t,y))\,v_k(t,x)\, K_u(t;x,y)\,\dd x\dd y, \\
J_{gap,k}(t) &= \iint_{B_1\times B_1}[\varphi_k(t,y)-v_e(t,y)]_+\,v_k(t,x)\, K_u(t;x,y)\,\dd x\dd y, \\
J_{r,k}(t) &= \iint _{B_1\times B_1} (v_k(t,x)-v_k(t,y))\,v_k(t,x)\, K_u(t;x,y)\,\dd x\dd y .
\end{aligned}
\end{equation}
The terms $J_{gap,k}$ and $J_{r,k}$ are nonnegative.
\end{lemma}

In the following sections, we estimate the four $J$ terms in \eqref{eq:defJgapk} pointwise in time. The arguments are in most cases identical to those in the elliptic case. In some cases, we are able to provide a refined inequality compared with the elliptic ones.

\subsection{\texorpdfstring{Estimating $J_{\mathrm{Tail},k}$}{Estimating J(Tail,k)}}

\begin{lemma} \label{l:tails_of_kernel}
Under assumptions (2)--(3) of Lemma~\ref{l:iteration}, for $\delta_1$ sufficiently small, we have
\begin{align}
\sup_{t \in (-1, 0]}\ \sup_{x \in B_{3/4}}\ \int_{B_1^c} K_u(t;x,y) \dd y &\leq C_T, \label{eq:kernel-tail-1}\\
\sup_{t \in (-1, 0]}\ \sup_{x \in B_{3/4}}\ \int_{B_1^c} \big( (2|y|)^{\alpha_0}-1\big)\, K_u(t;x,y) \dd y &\leq \eps(\delta_1), \label{eq:kernel-tail-2}
\end{align}
where $\alpha_0 = -\log(1-\delta_1)/\log2$, the constant $C_T$ depends only on $d,s$ and $p$, and $\eps(\delta_1) \to 0$ as $\delta_1 \to 0$. 
\end{lemma}

\begin{proof}
Recall $w(t,x)=u(t,x)-u(t,0)$.
For $x \in B_{3/4}$ and $t \in (-1, 0]$, using that $K_u = K_w$, $|w(t,x)|\leq 2$ (from \eqref{eq:dyadic_gradient_w} with $k=0$), and $|x-y| \ge |y|/4$ for $y \in B_1^c$, we obtain
\begin{align*}
    \int_{B_1^c} K_u(t;x,y) \dd y &\leq C \int_{B_1^c} \frac{(|w(t,y)|+1)^{p-2}}{|y|^{d+sp}} \dd y  \\
    & \leq C \left( {\Tail}_{p-2,sp}(w(t,\cdot),1)^{p-2} +1 \right) \\
    & \leq C\left( {\Tail}_{p-1,sp+1}(w(t,\cdot),1)^{p-2} +1 \right),
\end{align*}
for a constant $C>0$ depending only on $d, s$, and $p$, where in the last inequality we used Lemma~\ref{l:Tail_inequalities}.

The tail of $w$ is then bounded as follows. We choose $\delta_1$ small enough to ensure that $\alpha_0\leq\sigma/(2(p-1))$. Making use of \eqref{eq:dyadic_gradient_w} we obtain $|w(t,y)|\leq 1+C|y|^{1+\alpha_0}$ for $1<|y|<2^{K_0-1}$. We now bound the tail of $w$ splitting it at $2^{K_0-1}$ and using $\sup_{t \in (-1,0]} {\Tail}_{p-1,sp+1}(w(t,\cdot);2^{K_0-1}) \leq \eps_1\leq 1$ we obtain
\[
\int_{B_1^c}\frac{|w(t,y)|^{p-1}}{|y|^{d+sp+1}}\dd y
\leq C\left(\int_1^\infty r^{-1-\sigma+(p-1)\alpha_0}\dd r+\eps_1^{p-1}\right)
\leq C,
\]
for some $C>0$ depending only on $d, \, s$ and $p$. Combining this with the previous bound yields \eqref{eq:kernel-tail-1}.

To establish \eqref{eq:kernel-tail-2}, we multiply the pointwise estimate $K_u(t;x,y) \leq C(|w(t,y)|+1)^{p-2}|y|^{-d-sp}$ by the weight $\bigl((2|y|)^{\alpha_0}-1\bigr)$. Applying H\"older's inequality with $p-1$ and $(p-1)/(p-2)$, we obtain
\begin{align*}
    &\int_{B_1^c} \bigl((2|y|)^{\alpha_0}-1\bigr) |w(t,y)|^{p-2} |y|^{-d-sp} \dd y \\
    &\quad \leq \left( \int_{B_1^c} \bigl((2|y|)^{\alpha_0}-1\bigr)^{p-1} |y|^{-d-\sigma} \dd y \right)^{\frac{1}{p-1}} 
    \left( \int_{B_1^c} \frac{|w(t,y)|^{p-1}}{|y|^{d+sp+1}} \dd y \right)^{\frac{p-2}{p-1}}.
\end{align*}
By the same arguments used to bound the tail of $w$ above, the second factor on the right-hand side is bounded by a constant depending only on $d, s$, and $p$.

The first integral tends to zero as $\alpha_0\to0$, that is as $\delta_1 \to 0$, by domination with $C|y|^{-d-\sigma/2}$. The other term is treated in the same way, since $sp\geq\sigma$. The bounds are uniform in $K_0$ and $\eps_1\leq1$.

\end{proof}

\begin{lemma}\label{l:bound_J_Tail}
For every $t \in (-1,0]$ and $k \geq 1$,
\[
J_{Tail,k}(t) \geq -\eps(\delta_1) \int_{B_1} v_k(t,x)\, \dd x - C_T\, a_k(t) \int_{B_1} v_k(t,x)\,\eta_k(x)\, \dd x ,
\]
with $\eps(\delta_1)$ and $C_T$ as in Lemma~\ref{l:tails_of_kernel}.
\end{lemma}
\begin{proof}
For $y \in B_1^c$ we have $v_e(t,y) \leq (2|y|)^{\alpha_0}$ by \eqref{eq:ve-growth}, while for $x \in B_1$, $\varphi_k(t,x) \geq 1 - a_k(t)\eta_k(x)$. Hence
\[
\varphi_k(t,x)-v_e(t,y) \geq -\big( (2|y|)^{\alpha_0} - 1 \big) - a_k(t)\,\eta_k(x) .
\]
Multiply by $v_k(t,x) K_u(t;x,y) \geq 0$ and integrate. Since $\supp v_k(t,\cdot) \subset B_{3/4}$ by \eqref{eq:vk-pointwise}, the first term contributes at least $-\eps(\delta_1)\int v_k$ by \eqref{eq:kernel-tail-2}, and the second at least $-C_T a_k(t)\int v_k\eta_k$ by \eqref{eq:kernel-tail-1}.
\end{proof}

With the choice $\eps_2 = \eps_0+\eps(\delta_1)$, plugging Lemma~\ref{l:bound_J_Tail} into \eqref{eq:main_ineq} gives the simplified energy inequality
\begin{equation} \label{eq:main_ineq_simplified}
    \frac{1}{2}\frac{\dd}{\dd t} \int_{B_1} v_k^2\,\dd x  + J_{\varphi,k}(t) + J_{gap,k}(t) + J_{r,k}(t) \leq \big(a_k'(t) + C_T\, a_k(t)\big) \int_{B_1} v_k\, \eta_k\, \dd x .
\end{equation}

\subsection{\texorpdfstring{Estimating $J_{\mathrm{gap},k}$}{Estimating J(gap,k)}}

The lower bounds for $J_{gap,k}$ and $J_{r,k}$ rest on the following mechanism, introduced in \cite{GJS}. If $x$ is a point where $v_k>0$, then $\nabla u(t,x)$ is close to $e$ by \eqref{eq:alignment}. Along a ray $x+\rho\nu$ with $|\nu\cdot e|>1/6$, as long as $\nabla u(t,\cdot)$ stays close to $e$ on most of the ray, the function $\rho \mapsto u(t,x+\rho\nu)$ moves at speed at least $1/10$ in a fixed direction, so $|u(t,x)-u(t,y)| \gtrsim |x-y|$ and the kernel is nondegenerate
\begin{equation*}
K_u(t;x,y) = (p-1)\,\frac{|u(t,x)-u(t,y)|^{p-2}}{|x-y|^{d+sp}} \geq c\,|x-y|^{-d-\sigma},
\end{equation*}
since $p-2-sp = -\sigma$. Hence, the kernel bound is exactly that of a uniformly elliptic operator of order $\sigma$. The exceptional radii, where closeness to $e$ fails, are controlled by a stopping-time construction. Throughout this section, $\Gamma := \{\nu\in S^{d-1} : |\nu\cdot e|>1/6\}$ denotes the set of directions of the cones \eqref{def:cone}. The following elementary lemma quantifies the mechanism. We state it with the bound $|f'|\leq2$, since all the rays used below stay inside $B_2$, where $|\nabla u|\leq(1-\delta_1)^{-1}\leq2$ by Assumption (2) of Lemma \ref{l:iteration} with $n=1$.

\begin{lemma} \label{l:elementary-1D}
Let $f : [0,\ell] \to \R$ be Lipschitz with
$|f'| \leq 2$ a.e., and
$|\{\rho \in (0,\ell) : f'(\rho) < 1/10\}| \leq c_0\, \ell$ where $c_0 := 1/42$.
Then $f(\ell) - f(0) \geq \ell/20$.
\end{lemma}

\begin{proof}
It is just the fundamental theorem of calculus. See \cite[Lemma 4.9]{GJS}.
\end{proof}

The following result is a lower bound for the degenerate linearized kernel and plays the role of a coercivity bound. It is also an improvement on the bound obtained in the elliptic case in \cite[Lemma 4.10]{GJS}.

\begin{lemma}[Cone coercivity]\label{l:cone-coercivity}
Let $t\in(-1,0]$ and assume $\|\nabla u(t,\cdot)\|_{L^\infty(B_1)}\leq 1$. Let $D \subset B_1$ be convex, $x \in \overline D$, and let $S \subset D$ be measurable with $\dist(x,S)>0$ and
\begin{equation}\label{eq:cone-alignment-hyp}
\nabla u(t,y) \in B_{1/15}(e) \qquad \text{for every } y \in D\setminus S .
\end{equation}
Then
\[
\int_{S \cap C(x)} K_u(t;x,y) \dd y \;\geq\; c_A\, |S\cap C(x)| ,
\]
where $0<c_A\leq 1$ is a constant depending only on $d, \, s$ and $p$.
\end{lemma}

\begin{proof}
Since $S \subset D \subset B_1$ and $x\in\overline B_1$, every point of $S$ lies at distance less than $2$ from $x$. So we can write
\[
S\cap C(x) = \{ x+\rho\nu \ :\ \nu\in\Gamma,\ \rho\in S_\nu \}, \qquad S_\nu := \{\rho>0 : x+\rho\nu \in S\} .
\]
Since $D$ is convex and $x\in\overline D$, the radii $\rho$ for which $x+\rho\nu\in D$ form an interval $(0,\rho_\nu)$ (possibly empty), and $S_\nu \subset (0,\rho_\nu)$ with $\rho_\nu\leq2$.

Fix $\nu \in \Gamma$ with $S_\nu \neq \emptyset$, write $h_\nu(\ell) := |S_\nu \cap (0,\ell)|$, and define the stopping radius
\[
\ell_\nu := \inf\big\{ \ell \in (0,\rho_\nu] : h_\nu(\ell) \geq c_0\,\ell \big\},
\qquad \ell_\nu := \rho_\nu \ \text{ if the set is empty}.
\]
Since $h_\nu(\ell)=0$ for $\ell \leq \dist(x,S)$, we have $\ell_\nu>0$. The function $h_\nu$ is $1$-Lipschitz, and $h_\nu(\ell)<c_0\ell$ for $\ell<\ell_\nu$. Moreover, by continuity,
\begin{equation}\label{eq:stopping-mass}
h_\nu(\ell) \leq c_0\,\ell \quad \text{for every } \ell\leq\ell_\nu,
\qquad\text{and}\qquad
h_\nu(\ell_\nu) = c_0\,\ell_\nu \ \text{ if the infimum is over a nonempty set.}
\end{equation}

We can now bound the kernel along the ray. Let $\ell \in (0,\ell_\nu]$ and consider $f(\rho) := u(t,x+\rho\nu)$ on $[0,\ell]$. The segment lies in $\overline D \subset \overline B_1$, so $|f'|\leq1$ a.e. For $\rho\in(0,\ell)\setminus S_\nu$ the point $x+\rho\nu$ belongs to $D\setminus S$, so $\nabla u(t,x+\rho\nu) \in B_{1/15}(e)$ by \eqref{eq:cone-alignment-hyp}; if $\nu\cdot e>1/6$ this gives $f'(\rho) = \nabla u\cdot\nu \geq \nu\cdot e - \tfrac1{15} \geq \tfrac16-\tfrac1{15} = \tfrac1{10}$, and if $\nu\cdot e<-1/6$ the same holds for $-f$. The exceptional set $\{\rho\in(0,\ell): |f'(\rho)| <1/10\}$ is thus contained in $S_\nu\cap(0,\ell)$, of measure at most $c_0\ell$ by \eqref{eq:stopping-mass}. Lemma~\ref{l:elementary-1D}, applied to $\pm f$, gives $|u(t,x+\ell\nu)-u(t,x)|\geq\ell/20$, whence
\begin{equation}\label{eq:kernel-lower-on-ray}
K_u(t;x,x+\ell\nu) \geq c\,\ell^{\,p-2-(d+sp)} \qquad \text{for every } \ell \in (0,\ell_\nu] .
\end{equation}

We now integrate the kernel over the intersection of the ray with $S$. Set
\[
I_\nu := \int_{S_\nu\cap(0,\ell_\nu)} K_u(t;x,x+\ell\nu)\,\ell^{d-1}\dd\ell
\ \geq\ c\int_{S_\nu\cap(0,\ell_\nu)} \ell^{\,p(1-s)-3}\dd\ell ,
\]
by \eqref{eq:kernel-lower-on-ray}. We claim that for some $c_1>0$ depending on $d,\, s$ and $p$ we have 
\begin{equation}\label{eq:direction-contribution}
I_\nu \geq c_1\, |S_\nu|.
\end{equation}
If the infimum defining $\ell_\nu$ is over a nonempty set, then, using \eqref{eq:stopping-mass} and that the exponent $p(1-s)-3<-1$,
\[
I_\nu \geq c\, h_\nu(\ell_\nu)\, \ell_\nu^{\,p(1-s)-3} = c\,c_0\, \ell_\nu^{\,p(1-s)-2} \geq c\,c_0\, 2^{\,p(1-s)-2} \geq \tfrac{c\,c_0}{4} \geq \tfrac{c\,c_0}{8}\,|S_\nu| ,
\]
since $\ell_\nu\leq2$, the exponent $p(1-s)-2$ is negative, and $|S_\nu|\leq\rho_\nu\leq2$. If instead the set is empty, then $\ell_\nu = \rho_\nu$, so $S_\nu\cap(0,\ell_\nu) = S_\nu$ and
\[
I_\nu \geq c\int_{S_\nu} \ell^{\,p(1-s)-3}\dd\ell \geq c\,|S_\nu|\, 2^{\,p(1-s)-3} \geq \tfrac{c}{8}\,|S_\nu| .
\]
This proves \eqref{eq:direction-contribution}.

We can now integrate this inequality over all directions in $\Gamma$. In fact, integrating in polar coordinates and using $\ell^{d-1}\leq 2^{d-1}$,
\[
\int_{S\cap C(x)} K_u(t;x,y)\dd y = \int_\Gamma \int_{S_\nu} K_u(t;x,x+\ell\nu)\,\ell^{d-1}\dd\ell\dd\nu
\geq \int_\Gamma I_\nu \dd\nu \geq c_1 \int_\Gamma |S_\nu| \dd\nu ,
\]
while
\[
|S\cap C(x)| = \int_\Gamma\int_{S_\nu} \ell^{d-1}\dd\ell\dd\nu \leq 2^{d-1}\int_\Gamma |S_\nu|\dd\nu .
\]
The lemma follows with $c_A:=\min\{2^{1-d}c_1,1,|B_1|^{-1}\}$.
\end{proof}

We now derive the two forms of the gap bound. Both assume the bound on the heights 
\begin{equation}\label{eq:kappa-heights-cap} 
\bar a = \kappa|B_1|, \quad \text{with } \kappa \leq \frac{c_A r^2}{16}. 
\end{equation}
Recall from \eqref{eq:alignment} and $c_A\leq1$ that then $|\nabla u(t,x)-e| \leq 2\sqrt{\bar a} \leq r$ at every point of every level set. 

The following is a refinement of \cite[Lemma 4.10]{GJS}. Note that the right-hand side depends on $g(t)$ linearly rather than quadratically.

\begin{lemma}[Gap bound, using $\A_r(t)$]\label{l:Bound_Jgap1}
Assume $r \leq 1/15$ and \eqref{eq:kappa-heights-cap}. For every $t \in (-1,0]$ and every $k \geq 1$,
\[
J_{gap,k}(t) \geq \frac{c_A}{4}\, r^2\, g(t) \int_{B_1} v_k(t,x)\, \dd x .
\]
\end{lemma}

\begin{proof}
Fix $(t,x)$ with $v_k(t,x)>0$, so that $x\in A_k(t) \subset \overline B_{(1+\bar \eps)r_1}(x_1) \subset B_{1/8}(x_1)$ and $\nabla u(t,x) \in B_r(e)$, by \eqref{eq:alignment} and \eqref{eq:kappa-heights-cap}.
For $y \in \A_r(t)$ since $|\nabla u(t,y)|\leq1$ and $|\nabla u(t,y)-e|\geq r$, we have
$v_e(t,y)  \leq 1-\frac{r^2}{2},$
while $\varphi_k(t,y) \geq 1-a_k(t) \geq 1 - r^2/8$ by \eqref{eq:kappa-heights-cap}. Hence
\begin{equation}\label{eq:reservoir-separation}
[\varphi_k(t,y)-v_e(t,y)]_+ \geq \frac{r^2}{4} \qquad \text{for } y \in \A_r(t).
\end{equation}

Apply Lemma~\ref{l:cone-coercivity} with $D = B_1$ and $S = \A_r(t)$. On $B_1\setminus\A_r(t)$, the gradient lies in $B_r(e)\subset B_{1/15}(e)$, and $\dist(x,\A_r(t))>0$ because $\nabla u(t,\cdot)$ is continuous and $\nabla u(t,x)\in B_r(e)$. Since $x\in B_{1/8}(x_1)$ we have $C(x)\supset C^*(x_1)$, so $|\A_r(t)\cap C(x)| \geq g(t)$ (recall Lemma \ref{l:pickaball}) and
\[
\int_{\A_r(t)\, \cap\, C(x)} K_u(t;x,y) \dd y \geq c_A\, g(t) .
\]
Restricting the $y$-integral in $J_{gap,k}$ to $\A_r(t)\cap C(x)$ and using \eqref{eq:reservoir-separation},
\[
J_{gap,k}(t) \geq \frac{r^2}4 \int_{B_1} v_k(t,x) \left( \int_{\A_r(t)\, \cap\, C(x)} K_u(t;x,y) \dd y \right) \dd x
\geq \frac{c_A\, r^2}4\, g(t) \int_{B_1} v_k \dd x . \qedhere
\]
\end{proof}

The second form uses, instead of $\A_r(t)$, the set where $v_e$ is \emph{below the previous level}. It has no elliptic counterpart in \cite{GJS}; it plays here the role of the ``good term'' in the parabolic De Giorgi iteration of \cite{chihinchan}, adapted to the degenerate kernel $K_u$. Since level-set information is only useful inside the set $\{\eta_0=1\}$, where the barriers are separated by $(1-\lambda)a_{k-1}$, the lemma is localized there. The convex domain in Lemma~\ref{l:cone-coercivity} is $\{\eta_0=1\}=\overline B_{r_1}(x_1)$ itself.

The following is another estimate for $J_{gap,k}$ depending on the level sets $A_{k-1}(t)$ instead of the set $\A_r(t)$. It is a straightforward variant of Lemma~\ref{l:Bound_Jgap1}.

\begin{lemma}[Gap bound, using $A_{k-1}(t)$]\label{l:gap-below}
Let $2\leq k\leq N$, $t\in(-1,0]$, $\mu_0>0$, and assume \eqref{eq:kappa-heights-cap}. Suppose that 
\begin{equation}\label{eq:below-cone-hypothesis}
\big| C(x) \cap \big( B_{r_1/2}(x_1)\setminus A_{k-1}(t)\big) \big| \geq \mu_0 \qquad \text{for every } x \in \overline B_{r_1}(x_1).
\end{equation}
Then,
\[
J_{gap,k}(t) \geq c_A\, (1-\lambda)\, a_{k-1}(t)\, \mu_0 \int_{B_{r_1}(x_1)} v_k(t,x)\, \dd x .
\]
\end{lemma}

\begin{proof}
Fix $t$ and $x \in \overline B_{r_1}(x_1)$ with $v_k(t,x)>0$, and set
\[
D := B_{r_1}(x_1), \qquad S := B_{r_1}(x_1)\setminus A_{k-1}(t) .
\]
We check the hypotheses of Lemma~\ref{l:cone-coercivity}. In $D\setminus S = B_{r_1}(x_1)\cap A_{k-1}(t)$, we have $v_e > \varphi_{k-1} \geq 1-\bar a$, so \eqref{eq:alignment} and \eqref{eq:kappa-heights-cap} give $|\nabla u - e| \leq 2\sqrt{\bar a} < 1/15$. Moreover $x \in A_k(t) \subset A_{k-1}(t)$, and $A_{k-1}(t)$ is relatively open in $B_1$ (since $v_e(t,\cdot)$ and $\varphi_{k-1}(t,\cdot)$ are continuous), so a neighborhood of $x$ is disjoint from $S$, implying $\dist(x,S)>0$.

For $y\in S$ we have $\eta(y)=1$ and $v_e(t,y) \leq \varphi_{k-1}(t,y)$, hence
\[
[\varphi_k(t,y)-v_e(t,y)]_+ \geq \varphi_k(t,y)-\varphi_{k-1}(t,y) = \big(a_{k-1}(t)-a_k(t)\big)\,\eta(y) = (1-\lambda)\, a_{k-1}(t) .
\]

Since $B_{r_1/2}(x_1)\setminus A_{k-1}(t) \subset S$, hypothesis \eqref{eq:below-cone-hypothesis} gives $|S\cap C(x)| \geq \mu_0$, and Lemma~\ref{l:cone-coercivity} yields
\[
\int_{S\,\cap\, C(x)} K_u(t;x,y)\dd y \geq c_A\,\mu_0 .
\]
Restricting the $y$-integral in $J_{gap,k}$ to $S\cap C(x)$ and the $x$-integral to $B_{r_1}(x_1)$, we conclude
\[
J_{gap,k}(t) \geq (1-\lambda)\,a_{k-1}(t)\, c_A\,\mu_0 \int_{B_{r_1}(x_1)} v_k(t,x) \dd x . \qedhere
\]
\end{proof}

\subsection{\texorpdfstring{Estimating $J_{\varphi,k}$}{Estimating J(phi,k)}}

By the symmetry of $K_u$ in $(x,y)$,
\begin{equation*}
J_{\varphi,k}(t) = \frac{a_k(t)}{2} \iint_{B_1 \times B_1}(\eta_k(y)-\eta_k(x))(v_k(t,x)-v_k(t,y)) K_u(t;x,y)\,\dd x\dd y.
\end{equation*}

Our bound for $J_{\varphi,k}$ is the same one that we used in the elliptic case \cite[Lemma 4.11]{GJS} with $\delta_k$ replaced by $a_k(t)$. We state it in the next lemma.

\begin{lemma} \label{l:bound_Jphi}
For every $t \in (-1,0]$ and every $k \geq 1$,
\[
|J_{\varphi,k}(t)| \leq C\, a_k(t)\, \|\eta_k\|_{Lip}\, |A_k(t)|^{1/2}\, \sqrt{J_{r,k}(t)} .
\]
Consequently, by Young's inequality, for the preliminary levels $1\leq k\leq N$,
\begin{equation}\label{eq:fixed-cutoff-Jphi}
|J_{\varphi,k}(t)|\leq \frac12 J_{r,k}(t)+ C \bar \eps^{-2}\, a_k(t)^2\,|A_k(t)| ,
\end{equation}
 and for the nonlinear levels $k=N+j$, $j\geq1$,
\begin{equation}\label{eq:shrinking-cutoff-Jphi}
|J_{\varphi,N+j}(t)| \leq \frac12 J_{r,N+j}(t)+C\,2^{4j}\,a_{N+j}(t)^2\,|A_{N+j}(t)|,
\end{equation}
where $C$ denotes a universal constant.
\end{lemma}

\subsection{\texorpdfstring{Estimating $J_{r,k}$}{Estimating J(r,k)}}

Using again the symmetry of $K_u$,
\begin{equation*} 
J_{r,k}(t) = \frac{1}{2}\iint_{B_1 \times B_1}(v_k(t,x)-v_k(t,y))^2 K_u(t;x,y)\,\dd x\dd y .
\end{equation*}

The lower bound for $J_{r,k}$ is used in the nonlinear stage of the iteration only. The proof of the following result is as in \cite[Lemma 4.12]{GJS}, with $a_{N+j}(t)$ in place of the fixed height $\delta$. 

\begin{lemma} \label{l:bound_Jr}
Assume \eqref{eq:kappa-heights-cap}. For every $t \in (-1,0]$ and every $j\geq 1$,
\begin{align*}
J_{r,N+j}(t) &\geq c\, a_{N+j}(t)^2\, 2^{-2j}\, |A_{N+j+1}(t)|\; |A_{N+j}(t)|^{-\gs},
\end{align*}
for some universal constant $c>0$ depending only on $d, \, s, \, p$ and where $\gs := \frac{2-p(1-s)}{d} > 0$. Here we use the convention that the right-hand side is $0$ when $|A_{N+j+1}(t)|=0$.
\end{lemma}

\subsection{The preliminary phase}\label{s:preliminary_contraction}

In this section we run the preliminary phase. Here we consider $1\leq k\leq N$, $\eta_k = \eta_0$, the cutoff defined in Section \ref{s:DGsetup}, and the heights $a_k = \lambda^{k-1}a_1$ with $a_1$ defined by \eqref{eq:ODEa}. The goal is to prove that the space-time measure of the superlevel set, after $N$ iterations, is sufficiently small. The main result is  Proposition~\ref{p:linear-preliminary-contraction}. We keep assuming the setting of Lemma~\ref{l:DGengine}.

\begin{proposition}[Small measure]\label{p:linear-preliminary-contraction}
 For any  $\omega_\star> 0$ small, with the choices of Remark~\ref{r:constants-order}, there exists an integer $N$ such that 
\[
 \int_{-1}^0 |A_{N+1}(t)|\dd t \leq \omega_\star .
\]

\end{proposition}

We begin by establishing the master inequality we use during the preliminary phase.
Hereafter we use the notation
\begin{equation*}
E_k(t) := \int_{B_1} v_k(t,x)^2 \dd x .
\end{equation*}

\begin{lemma}[Master inequality]\label{l:master-preliminary}
For every $1\leq k\leq N$ and almost every $t\in(-1,0]$,
\begin{equation}\label{eq:master-preliminary}
\frac12 E_k'(t) + \frac12 J_{gap,k}(t) + \frac12 J_{r,k}(t) \leq C_0 \bar \varepsilon^{-2}\, a_k(t)^2\, |A_k(t)|,
\end{equation}
for some $C_0>0$ depending on $d,\, s,$ and $p$.

\end{lemma}

\begin{proof}

The right-hand side of \eqref{eq:main_ineq_simplified} consists of two terms. The first term can be estimated by half of $J_{gap,k}$ using 
Lemma~\ref{l:Bound_Jgap1} and recalling \eqref{eq:ODEa} and \eqref{eq:kappa-heights-cap}. This gives 
\begin{align*}
    a_k'(t)\int_{B_1} v_k\,\eta_k \dd x \leq \lambda^{k-1} \kappa\, g(t) \int_{B_1} v_k\,\eta_k \dd x \leq \frac{c_A r^2}{8}\, g(t) \int_{B_1} v_k \dd x \leq \frac12\, J_{gap,k}(t).
\end{align*}
The second term comes from the tail term and using \eqref{eq:vk-pointwise} we can bound it as 
\[
C_T\, a_k(t) \int_{B_1} v_k\,\eta_k \dd x \leq C_T\, a_k(t)^2\, |A_k(t)| .
\]
Hence \eqref{eq:main_ineq_simplified} reduces to $\tfrac12 E_k' + J_{\varphi,k} + \tfrac12 J_{gap,k} + J_{r,k} \leq C_T a_k^2 |A_k|$, and \eqref{eq:fixed-cutoff-Jphi} bounds $-J_{\varphi,k} \leq \tfrac12 J_{r,k} + C \bar \eps^{-2} a_k^2|A_k|$, which yields \eqref{eq:master-preliminary}.
\end{proof}

The natural quantity in the preliminary phase is the normalized energy, defined by
\begin{equation*}
F_k(t) := \frac{E_k(t)}{a_k(t)^2} \qquad \text{for } t \in (\tilde{t}_a,0], \qquad \tilde{t}_a := \inf\{ t : a_1(t)>0 \} \in [t_a, t_b) .
\end{equation*}
Note that $\tilde{t}_a<t_b$ because $\int_{t_a}^{t_b}g \dd t>0$, and that $A_k(t)=\emptyset$, $E_k(t)=0$ for $t\leq \tilde{t}_a$. We rewrite the master inequality \eqref{eq:master-preliminary} in terms of $F_k$.

\begin{lemma}[Bounds for the normalized energy]\label{l:F-basic}
For every $1\leq k\leq N$, the function $F_k$ is well-defined on $(\tilde{t}_a,0]$, $F_k(t)\to0$ as $t\downarrow \tilde{t}_a$, and for almost every $t\in(\tilde{t}_a,0]$,
\begin{equation}\label{eq:Fk-ode}
F_k'(t) \leq 2\,C_0 \bar \varepsilon^{-2}\, |A_k(t)| - a_k(t)^{-2} J_{gap,k}(t) \leq 2\,C_0 \bar \varepsilon^{-2}\,|B_1| .
\end{equation}
Consequently,
\begin{equation}
\int_{\tilde{t}_a}^0 a_k(t)^{-2}\, J_{gap,k}(t) \dd t \leq 2\,C_0\bar \varepsilon^{-2}\,|B_1| . \label{eq:weighted-gap-budget}
\end{equation}
\end{lemma}

\begin{proof}
On $(\tilde{t}_a,0]$ we have $a_1>0$, since $a_1$ is nondecreasing. Since moreover $a_k'\geq0$ and $F_k\geq0$, after dropping $\tfrac12 J_{r,k}\geq0$ from the master inequality \eqref{eq:master-preliminary}, we obtain, for $t \in (\tilde{t}_a,0]$,
\[
F_k' = \frac{E_k'}{a_k^2} - 2\,\frac{a_k'}{a_k}\, F_k
\leq \frac{E_k'}{a_k^2}
\leq \frac{2C_0\, \bar \varepsilon^{-2} a_k^2 |A_k| - J_{gap,k}}{a_k^2} ,
\]
which is \eqref{eq:Fk-ode}. The final bound in \eqref{eq:Fk-ode} uses $|A_k|\leq|B_1|$ and $J_{gap,k}\geq0$.

For the limit at $\tilde{t}_a$, integrate $E_k' \leq 2C_0 \bar\eps^{-2} a_k^2|A_k|$ from $\tilde{t}_a$ (where $E_k=0$) to $t$, using that $a_k$ is nondecreasing
\[
E_k(t) \leq 2C_0 \bar \varepsilon^{-2} \int_{\tilde{t}_a}^t a_k(h)^2\,|A_k(h)|\dd h
\leq 2C_0\bar \varepsilon^{-2} \, a_k(t)^2\, \int_{\tilde{t}_a}^{t}|A_k(h)|\dd h ,
\]
so $F_k(t) \leq 2C_0\bar \varepsilon^{-2}  \int_{\tilde{t}_a}^t|A_k(h)| \dd h \to 0$ as $t \downarrow \tilde{t}_a$.

Finally, integrating \eqref{eq:Fk-ode} over $(\tilde{t}_a,0]$ and using $F_k\geq0$, $F_k(\tilde{t}_a^+)=0$, we obtain
\[
\int_{\tilde{t}_a}^0 a_k(t)^{-2} J_{gap,k}(t) \dd t \leq 2C_0\bar \varepsilon^{-2} \int_{\tilde{t}_a}^0 |A_k(t)| \dd t + F_k(\tilde{t}_a^+) - F_k(0) \leq 2\,C_0\bar \varepsilon^{-2}\,|B_1| . \qedhere
\]
\end{proof}

We can now state and prove the parabolic counterpart of the elliptic one-step measure drop \cite[Lemma 4.13]{GJS}. Here we use a mass drop dichotomy in the spirit of the second De Giorgi lemma of \cite{chihinchan}. The obstruction is that the gap coercivity acts only at past times, while the smallness of the level sets is needed at future times. The resolution is a dichotomy: \emph{either} the level sets are small in the future, \emph{or} a fixed amount of space-time measure sits strictly between two consecutive levels, and, the level sets being nested, this can only happen a bounded number of times.
We start with an elementary bound that fixes the choice of $\bar \eps$ in terms of the smallness threshold $\omega_\star$.

\begin{lemma}\label{l:cube}

For any $\omega_\star>0$ small there exists $\bar \varepsilon>0$ such that if
\begin{equation} \label{eq:pick_eps}
    |B_{(1+\bar \eps)r_1}\setminus B_{r_1}|\leq \frac{\omega_\star}{32},
\end{equation}
and $A\subset \supp \eta_0$ is a measurable set with $|A|\geq \omega_\star$, 
then
\[
\int_A \eta_0^2 \dd x\geq \frac{\omega_\star}{2}. 
\]

\end{lemma}

\begin{proof}
Note that, by Taylor expansion, there is $\bar \varepsilon>0$ sufficiently small, depending only on the dimension, such that
\[
|B_{(1+\bar \eps)r_1}\setminus B_{r_1}|\leq 2d \bar \varepsilon.
\]
For a dimensional constant $\bar c\leq 1/(64d) $, we further choose $\bar\varepsilon:= \bar c\omega_\star$ so that
\[
|B_{(1+\bar \eps)r_1}\setminus B_{r_1}|\leq \frac{\omega_\star}{32}.
\]

Since $A\subset \supp \eta_0$, we can write
\begin{align*}
    \int_A \eta_0^2\dd x\geq |A\cap B_{r_1}(x_1)| \geq |A|-|A\setminus B_{r_1}(x_1)|\geq |A|-|B_{(1+\bar \eps)r_1}\setminus B_{r_1}| \geq \frac{\omega_\star}{2}.
\end{align*}

\end{proof}

For $k\geq 1$, define the quantities
\begin{equation*}
\omega_k := \int_{-1}^0 |A_k(t)| \dd t.
\end{equation*}

\begin{proposition}[De Giorgi Lemma]\label{p:mass-transfer}

Let $2\leq k\leq N-1$. For any $\omega_\star>0$ small there exist $\gamma_\star$ and $\bar\eps,\lambda<1/4$ depending on $\omega_\star$ such that at least one of the following holds:
\begin{enumerate}[label=(\Alph*)]
\item \emph{(smallness in the future)}
$\displaystyle \int_{t_b}^0 |A_{k+1}(t)|\dd t \leq \omega_\star$;
\item \emph{(mass drop)}
$\displaystyle \omega_{k+1} \leq \omega_{k-1} - \gamma_\star$.
\end{enumerate}
\end{proposition}

Since $A_{k+1}(t)\subset A_{k-1}(t)$ for every $t$, alternative (B) states that a definite amount of space-time measure lies strictly between the levels $k-1$ and $k+1$.

\begin{proof}
Assume (A) fails; we prove (B). Since $t_b>-1$, there is $t_0 \in (t_b,0]$ with
\begin{equation*}
|A_{k+1}(t_0)| \geq \omega_\star .
\end{equation*}

We now show that  also the normalized energy is large at $t_0$. By \eqref{def_levelsetsAk}, for $x\in A_{k+1}(t)$ we have $v_k(t,x) \geq \varphi_{k+1}(t,x)-\varphi_k(t,x) = (1-\lambda)\,a_k(t)\,\eta_0(x)$. Hence, choosing $\bar \eps=\bar c \omega_\star$ depending on $\omega_\star$ and applying Lemma~\ref{l:cube}, we obtain a lower bound for any $t$ such that $|A_{k+1}(t)|\geq \omega_\star$. In particular, we obtain
\begin{equation}\label{eq:cube-chebyshev}
F_k(t_0) \geq (1-\lambda)^2  \int_{A_{k+1}(t_0)}\eta_0^2 \dd x \geq (1-\lambda)^2\, \frac{\omega_\star}{2}
\geq \frac{\omega_\star  }{4}\, ,
\end{equation}
using $\lambda\leq1/4$.

Recall $\tilde{t}_a := \inf\{ t : a_1(t)>0 \} \in [t_a, t_b)$
 and define the band

\[
D := \left\{ t \in (\tilde{t}_a, t_0] : \frac{\omega_\star}{16} < F_k(t) < \frac{\omega_\star}{8} \right\}.
\]

Since $F_k$ is absolutely continuous, $F_k(\tilde t_a^+)=0$ and \eqref{eq:cube-chebyshev} holds, the function $F_k$ crosses the band from below to above at least once. Also, since $\bar \eps=\bar c \omega_\star$, \eqref{eq:Fk-ode} gives 
\[
F_k'(t) \leq 2\,C_0\, {\bar c}^{\,- 2}\omega_\star^{-2}|B_1|.
\]
 Since the band has width $\omega_\star/16$ and $F_k$ must traverse it while its derivative is at most $2\,C_0\, {\bar c}^{\,- 2}\omega_\star^{-2}|B_1|$, we conclude
\begin{equation}\label{eq:band-length}
|D| \geq {\bar c}^{\, 2}\frac{\omega_\star^3}{32 C_0 |B_1|} .
\end{equation}

Notice that for $t\in D$ we have
\begin{equation}\label{eq:Ak+1-small-on-band}
|A_{k+1}(t)| \leq \omega_\star.
\end{equation}
Indeed, if this were not the case, then \eqref{eq:cube-chebyshev} would imply that $F_k(t) \geq \omega_\star/4$ contradicting that $t \in D$.

Define now the exceptional times
\[
\mathcal F := \big\{ t \in D \ :\ |A_{k-1}(t)\cap B_{r_1}(x_1)| \leq \mu_0 \big\},
\]
for some dimensional constant $\mu_0$ which will be later specified.
We claim
\begin{equation}\label{eq:F-small}
|\mathcal F| \leq \frac{|D|}{2} .
\end{equation}
To see this, we first check that, for $t\in\mathcal F$, the hypothesis \eqref{eq:below-cone-hypothesis} of Lemma~\ref{l:gap-below} holds. Fix $x\in\overline B_{r_1}(x_1)$ and suppose that $(x_1 - x)\cdot e \geq 0$ (otherwise, replace $e$ by $-e$ below). Consider the ball
\[
B^* := B_{r_1/16}\big(x_1 +\tfrac{7r_1}{16}\, e\big) \subset B_{r_1/2}(x_1) .
\]
Every $y \in B^*$ belongs to $C(x)$. Indeed $|y-x| \leq r_1 + r_1/2=3r_1/2 \leq 2$, and
\[
|(y-x)\cdot e| \geq (y-x_1)\cdot e + (x_1 - x)\cdot e \geq \tfrac{3r_1}8,
\]
which gives $|(y-x)\cdot e|>\frac{1}{6}|y-x|.
$
Hence we can write
\begin{align*}
    |C(x)\cap &\,(B_{r_1/2}(x_1)\setminus A_{k-1}(t))| \geq |B^*| - |A_{k-1}(t)\cap B_{r_1}(x_1)| \geq  16^{-d}|B_{r_1}| - \mu_0 \geq \mu_0 ,
\end{align*}
where we used $|B^*| = 16^{-d}|B_{r_1}|$, the definition of $\mathcal F$ and chose $\mu_0$ so that the last inequality holds. Lemma~\ref{l:gap-below} then gives, for $t\in\mathcal F$,
\[
J_{gap,k}(t) \geq c_A\,(1-\lambda)\,a_{k-1}(t)\,\mu_0 \int_{B_{r_1}(x_1)} v_k(t,x)\dd x
\geq \frac{c_A}{2}\,\frac{a_k(t)}{\lambda}\,\mu_0\, \frac{1}{a_k(t)} \int_{B_{r_1}(x_1)}v_k^2\dd x ,
\]
where we used $a_{k-1} = a_k/\lambda$, $1-\lambda\geq\tfrac12$, and $\int_{B_{r_1}} v_k \geq \int_{B_{r_1}}v_k^2 / \sup v_k \geq \int_{B_{r_1}}v_k^2/a_k$, by \eqref{eq:vk-pointwise}. Now we bound the energy on $B_{r_1}(x_1)$ from below at times on the band. By \eqref{eq:vk-pointwise}, the ring contribution satisfies
\begin{equation*}
E_k(t) - \int_{B_{r_1}(x_1)}v_k^2\dd x = \int_{\supp\eta_0\setminus B_{r_1}(x_1)} v_k^2 \dd x \leq a_k(t)^2\, |B_{(1+\bar \eps)r_1}\setminus B_{r_1}| \leq \frac{\omega_\star}{32}\,a_k(t)^2,
\end{equation*}
where in the last inequality we used the choice of $\bar \eps$ in \eqref{eq:pick_eps}.

Hence, for $t\in D$,
\[
\frac{1}{a_k(t)^2} \int_{B_{r_1}(x_1)}v_k^2\dd x \geq F_k(t) - \frac{\omega_\star}{32} \geq \frac{\omega_\star}{16} - \frac{\omega_\star}{32} = \frac{\omega_\star}{32} ,
\]
and therefore, for $t \in \mathcal F$,
\[
a_k(t)^{-2}\, J_{gap,k}(t) \geq \frac{c_A\,\mu_0}{2\lambda}\, \frac{1}{a_k(t)^{2}} \int_{B_{r_1}(x_1)}v_k^2\dd x\geq \frac{c_A\,\mu_0\, \omega_\star}{64\,\lambda} .
\]
Integrating over $\mathcal F$ and using \eqref{eq:weighted-gap-budget} we get
\[
|\mathcal F|\cdot \frac{c_A\,\mu_0\,\omega_\star}{64\,\lambda} \leq 2\,C_0\,{\bar c}^{\,- 2}\omega_\star^{-2}|B_1| ,
\qquad\text{so}\qquad
|\mathcal F| \leq \frac{64\,\lambda\, 2\,C_0\,|B_1| }{c_A\,\mu_0\, {\bar c}^{\, 2} \,\omega_\star^3} \leq {\bar c}^{\, 2}\frac{\omega_\star^3}{64 C_0 |B_1|} ,
\]
provided we make the choice 
\begin{equation}\label{def:lambda}
    \lambda :=  \frac{c_A\,{\bar c}^{\, 4} \omega_\star^6\, \mu_0\, }{256 \cdot 32\,C_0^2\,|B_1|^2 }<\frac{1}{4}. 
\end{equation} 
This, combined with \eqref{eq:band-length}, proves \eqref{eq:F-small}.

For $t \in D\setminus\mathcal F$ we have $|A_{k-1}(t)| \geq |A_{k-1}(t)\cap B_{r_1}(x_1)| > \mu_0$ and, taking  $\omega_\star \leq \mu_0/2$, we deduce from \eqref{eq:Ak+1-small-on-band} that $|A_{k+1}(t)| \leq \mu_0/2$. This implies
\[
|A_{k-1}(t)| - |A_{k+1}(t)| \geq \frac{\mu_0}{2} .
\]
Since the level sets are nested, $|A_{k-1}(t)| - |A_{k+1}(t)| \geq 0$ for every $t$, and therefore, using \eqref{eq:band-length} and \eqref{eq:F-small},
\[
\omega_{k-1} - \omega_{k+1} = \int_{-1}^0 \big( |A_{k-1}(t)| - |A_{k+1}(t)| \big) \dd t \geq \int_{D\setminus\mathcal F} \big( |A_{k-1}(t)| - |A_{k+1}(t)| \big) \dd t \geq \frac{|D|}{2}\cdot\frac{\mu_0}{2} \geq  \gamma_\star ,
\]
by choosing $\gamma_\star$ such that
\[
 \gamma_\star\leq  \frac{{\bar c}^{\, 2}\omega_\star^3}{64 C_0 |B_1|}\cdot\frac{\mu_0}{2}.
\]
This implies alternative (B) and concludes the proof.

\end{proof}

Now we present the proof of Proposition \ref{p:linear-preliminary-contraction}.

\begin{customproof}{Proposition \ref{p:linear-preliminary-contraction}}
 The choice of $\omega_\star$ defines $\gamma_\star$ through Proposition  ~\ref{p:mass-transfer}. Let 
\begin{equation} \label{def:N}
     N := 2\left\lceil \frac{|B_1|}{\gamma_\star} \right\rceil + 4,
\end{equation}
and set $M := \lceil |B_1|/\gamma_\star\rceil + 1$ and apply Proposition~\ref{p:mass-transfer} at the even levels $k = 2, 4, \dots, 2M$; note $2M \leq N-1$. If alternative (B) held at every such level, then summing and telescoping,
\[
|B_1| \geq \omega_1 \geq \omega_1 - \omega_{2M+1} = \sum_{j=1}^{M} \big( \omega_{2j-1} - \omega_{2j+1} \big) \geq M\,\gamma_\star > |B_1| ,
\]
a contradiction. Therefore alternative (A) holds at some even level $k_\star \leq N-1$, that is,
\[
\int_{t_b}^0 |A_{k_\star+1}(t)|\dd t \leq \omega_\star .
\]
By the nesting \eqref{eq:vk-pointwise}, $A_{N+1}(t) \subset A_{k_\star+1}(t)$ for every $t$ (note $k_\star+1\leq N$ and the nesting holds across the transition by Lemma~\ref{l:a-bounds}(1)). Moreover $A_{N+1}(t) = \emptyset$ for $t \leq t_1 = t_b$, since $a_{N+1}$ vanishes there. Hence
\[
\omega_{N+1} = \int_{t_b}^0 |A_{N+1}(t)|\dd t \leq \int_{t_b}^0 |A_{k_\star+1}(t)|\dd t \leq \omega_\star . \qedhere
\]
\end{customproof}

\subsection{The nonlinear recursion}\label{s:nonlinear}

The second phase is the standard nonlinear De Giorgi recursion, run on the levels $N+j$, $j\geq1$. This phase is driven by the coercive bound of Lemma~\ref{l:bound_Jr}.

Plugging \eqref{eq:shrinking-cutoff-Jphi} into \eqref{eq:main_ineq_simplified}, dropping $J_{gap,N+j}\geq0$ and using Lemma~\ref{l:bound_Jr} and $\eta_{N+j}\leq 1$  we obtain the following energy inequality
\begin{equation} \label{eq:combined_energy_ineq}
\begin{aligned}
&\frac{1}{2}\frac{\dd}{\dd t} \int_{B_1} v_{N+j}^2\dd x
+ c\, a_{N+j}(t)^2 2^{-2j} |A_{N+j+1}(t)||A_{N+j}(t)|^{-\gs}
\\
&\quad\leq C 2^{4j} a_{N+j}(t)^2 |A_{N+j}(t)| 
 + \big(a_{N+j}'(t)+C_T a_{N+j}(t)\big)\!\int_{B_1}\! v_{N+j} \dd x .
\end{aligned}
\end{equation}

\begin{lemma} \label{l:omega-decay-from-start}
For any integer $j\geq 1$, we have
\begin{equation} \label{eq:omega-recursion-shifted}
    \omega_{N+j+1} \leq C 2^{6j} \omega_{N+j}^{1+\beta},
\end{equation}
for some constant $C$, depending on $d$, $s$, $p$, and $|t_b|$, and where $\beta := \frac{\gs}{1+\gs}>0$ with $\gs = \frac{2-p(1-s)}{d} > 0$.
\end{lemma}
\begin{proof}
    Set $m_j(t):=|A_{N+j}(t)|$, and $E_j(t):=\int_{B_1}v_{N+j}(t,x)^2\dd x$. Then $\omega_{N+j} = \int_{-1}^0 m_j(t) \dd t$.
    By Lemma~\ref{l:a-bounds}, for a.e.\ $t$, we have
    \begin{equation}
     \big(a_{N+j}'(t)+C_T a_{N+j}(t)\big)\!\int_{B_1}\! v_{N+j} \dd x \leq \big(a_{N+j}'(t)+C_T a_{N+j}(t)\big) a_{N+j}(t) m_j(t) \leq C \delta_0^2 2^j m_j(t)
    \end{equation}
    where $C$ depends also on $|t_b|$.
Thus we can estimate the right-hand side of \eqref{eq:combined_energy_ineq} by $C 2^{4j} \delta_0^2 m_j(t)$. For the left-hand side of \eqref{eq:combined_energy_ineq} we use that for all $t$
\begin{equation} \label{eq:ak-lower-on-next-set-shifted}
a_{N+j}(t)^2\, m_{j+1}(t)\geq \delta_0^2\, m_{j+1}(t),
\end{equation}
which follows from \eqref{eq:prop-AN}.
These bounds reduce \eqref{eq:combined_energy_ineq} to
\begin{equation} \label{eq:DG-differential-recursive-shifted}
\frac12 E_j'(t)
+c\,\delta_0^2\, 2^{-2j}\,m_{j+1}(t)\,m_j(t)^{-\gs}
\leq C 2^{4j}\,\delta_0^2\,m_j(t) \quad \text{for a.e.\ } t .
\end{equation}

Since $v_{N+j}(-1,\cdot)=0$, integrating the above inequality from $-1$ to $t$ and dropping the coercive term yields
\begin{equation} \label{eq:Ek-sup-bound-shifted}
E_j(t)
\leq C\,2^{4j}\,\delta_0^2\,\omega_{N+j}.
\end{equation}
Combining the Chebyshev estimate of Lemma~\ref{l:chebi} with \eqref{eq:ak-lower-on-next-set-shifted} gives
$c2^{-2j}\delta_0^2 m_{j+1}(t)\leq E_j(t)$, which, coupled with \eqref{eq:Ek-sup-bound-shifted}, leads to
\begin{equation} \label{eq:est_point_mj+1}
    \sup_{t \in (-1,0]} m_{j+1}(t) \leq C 2^{6j} \omega_{N+j}.
\end{equation}
On the other hand integrating \eqref{eq:DG-differential-recursive-shifted} over $(-1,0)$ and using $E_j(0)\geq0$ yields
\begin{equation} \label{eq:weighted-measure-bound-shifted}
\int_{-1}^0 m_{j+1}(t)\,m_j(t)^{-\gs}\,\dd t
\leq C\,2^{6j}\,\omega_{N+j}.
\end{equation}
On the set where $m_{j+1}(t)>0$, we split $m_{j+1}$ as follows:
\[
m_{j+1}
=m_{j+1}^{\gs/(1+\gs)}
\big(m_{j+1}\,m_j^{-\gs}\big)^{1/(1+\gs)}
\,m_j^{\gs/(1+\gs)},
\]
and estimate $\omega_{N+j+1}$ by using \eqref{eq:est_point_mj+1} in the first factor above and applying H\"older's inequality with exponents $1+\gs$ and $ \tfrac{1+\gs}{\gs}$ to the remaining two factors. Using \eqref{eq:weighted-measure-bound-shifted}, this leads to
\begin{align*}
\omega_{N+j+1}
&\leq
\left(\sup_{t\in (-1,0]} \,m_{j+1}\right)^{\frac{\gs}{1+\gs}}
\left(\int_{-1}^0 m_{j+1}\,m_j^{-\gs}\,\dd t\right)^{\frac{1}{1+\gs}}
\left(\int_{-1}^0 m_j\,\dd t\right)^{\frac{\gs}{1+\gs}} \\
&\leq
\big(C2^{6j}\,\omega_{N+j}\big)^{\frac{\gs}{1+\gs}}
\big(C2^{6j}\,\omega_{N+j}\big)^{\frac{1}{1+\gs}}
\;\omega_{N+j}^{\frac{\gs}{1+\gs}}
= C\,2^{6j}\,\omega_{N+j}^{1+\beta},
\end{align*}
which proves \eqref{eq:omega-recursion-shifted}.

\end{proof}

The inequality \eqref{eq:omega-recursion-shifted} in Lemma \ref{l:omega-decay-from-start} is the typical recurrence relationship in De Giorgi's iteration. Together with the smallness of $\omega_{N+1}$ given by Proposition \ref{p:linear-preliminary-contraction}, it implies that $\omega_k \to 0$ as $k \to \infty$. This is the result of an elementary convergence result for sequences that we recall below.

\begin{lemma}\label{l:fast_convergence}
    Let $(Y_j)_{j}$ be a sequence of positive real numbers satisfying the recursive inequalities
    \[
        Y_{j+1}\leq Cb^{j} Y_j^{1+\beta},
    \]
    where $C,b>1$ and $\beta>0$ are given numbers. If
    \[
        Y_0\leq C^{-1/\beta}b^{-1/\beta^2},
    \]
    then $Y_j\to 0$ as $j\to\infty$.
\end{lemma}

The proof of Lemma \ref{l:DGengine} follows by simply putting together our lemmas above.

\begin{customproof}{Lemma \ref{l:DGengine}}

Fix the constants as in Remark~\ref{r:constants-order} and let $\delta_0$ be given by \eqref{def:delta-from-aN}. 
Let $Y_j := \omega_{N+1+j}$ for $j \geq 0$. By Lemma~\ref{l:omega-decay-from-start}, this sequence satisfies $Y_{j+1} \leq C b^{\,j} Y_j^{1+\beta}$ with $b := 2^6$. 
By setting $$\omega_\star:=\min\{C^{-1/\beta} b^{-1/\beta^2},\mu_0/2\},$$ from Lemma~\ref{l:fast_convergence}, if $Y_0 \leq \omega_\star$, then $Y_j \to 0$. This condition is precisely the threshold given by Proposition ~\ref{p:linear-preliminary-contraction}. 
Thus we have $\omega_{N+j} \to 0$ as $j \to \infty.$

 For $t\geq t_b/2$ we have  $a_{N+j}(t)=\theta_j\delta_0 \geq \delta_0$ by \eqref{def:nonlinear-heights} and $\eta_{N+j} \equiv 1$ on $B_{1/16}(x_1)$ for $j\geq N+1$. Combining this with $\omega_{N+j} \to 0$ as $j \to \infty$, by \eqref{eq:eps2-small}, leads to 
\[
v_e(t,x) \leq 1-\delta_0+\eps_2 \leq 1-\frac{\delta_0}{2}
\qquad\text{for } (t,x) \in (t_b/2,0]\times B_{1/16}(x_1).
\]
This proves the lemma with $\delta_2:=\delta_0/2$.
\end{customproof}

The proof of Lemma~\ref{l:iteration} follows from subsequent applications of Lemma~\ref{l:DGengine}. The argument is similar to that in \cite[Lemma~4.1]{GJS}, except that we move forward in time with each application of Lemma~\ref{l:DGengine}.

\begin{customproof}{Lemma \ref{l:iteration}}
    
We may assume $r\leq1/15$, since decreasing $r$ enlarges $\A_r(t)$.
Let $x_1\in\{e/2,-e/2\}$ be given by Lemma~\ref{l:pickaball}.
Applying Lemma~\ref{l:DGengine} with $(t_a,t_b)=(-1,-2\tau)$ and
mass $\mu/2$, we obtain, for some $0<\delta_2\leq1/15$,
\begin{equation*}
 e\cdot\nabla u\leq1-\delta_2
 \quad\text{in }(-\tau,0]\times B_{1/16}(x_1).
\end{equation*}
In particular, $B_{1/16}(x_1)\subset\A_{\delta_2}(t)$ for $t> -\tau$.
Since $B_{1/16}(x_1)\subset C^*(-x_1)$, it follows that
\[
 \int_{-\tau}^{-\tau/2}
 |\A_{\delta_2}(t)\cap C^*(-x_1)|\,\dd t
 \geq\frac\tau2|B_{1/16}|.
\]
We apply Lemma~\ref{l:DGengine} again, now with center $-x_1$,
$r=\delta_2$ and $(t_a,t_b)=(-\tau,-\tau/2)$.
Decreasing $\delta_2$ if necessary, we obtain
\begin{equation}\label{eq:second-ball}
 e\cdot\nabla u\leq1-\delta_2
 \quad\text{in }(-\tau/4,0]\times
 \bigl(B_{1/16}(x_1)\cup B_{1/16}(-x_1)\bigr).
\end{equation}

Now let $z\in\overline B_{1/2}$. If $z\cdot x_1\leq0$, then for
$x\in B_{1/8}(z)$ and $y\in B_{1/16}(x_1)$ we have
\[
 |(y-x)\cdot e|>\frac5{16},
 \qquad |y-x|<2.
\]
This gives $B_{1/16}(x_1)\subset C^*(z)$.
If $z\cdot x_1>0$, the same argument applies with $-x_1$. Thus
\begin{equation*}
 B_{1/16}(x_1)\subset C^*(z)
 \quad\text{or}\quad B_{1/16}(-x_1)\subset C^*(z).
\end{equation*}
Combining this with \eqref{eq:second-ball}, we obtain
\[
 \int_{-\tau/4}^{-\tau/8}
 |\A_{\delta_2}(t)\cap C^*(z)|\,\dd t
 \geq\frac\tau8|B_{1/16}|.
\]
A final application of Lemma~\ref{l:DGengine}, with $r=\delta_2$ and
$(t_a,t_b)=(-\tau/4,-\tau/8)$, gives, after decreasing $\delta_2$,
\[
 e\cdot\nabla u\leq1-\delta_2
 \quad\text{in }(-\tau/16,0]\times B_{1/16}(z).
\]
The constants are independent of $z$. Taking a finite cover of
$\overline B_{1/2}$ by these balls, all with the same interval $(-\tau/4,-\tau/8)$, proves the desired estimate.

\end{customproof}

\section{Controlling the oscillation in time}\label{s:oscillation_time}

Our plan is to apply Lemma \ref{l:iteration} in shrinking cylinders, with the goal of reducing the size of the gradient at smaller scales. If we succeed in applying the lemma for every vector $e \in S^{d-1}$, we will have reduced the size of the gradient in a smaller cylinder. We will then rescale and repeat the argument. The conclusion of the lemma immediately gives the appropriate bound for $|\nabla u|$ after rescaling. 
 The gradient bound controls $w(t,x)=u(t,x)-u(t,0)$ in space. We now use the equation to control its oscillation in time, with constants depending only on its tail.
As before, we assume $u$ is a weak solution which is smooth in $x$.

\begin{lemma}\label{l:centered-time}
Let $u$ be a solution of \eqref{eq:fracplap} in $Q_{1,1}$, with $|\nabla u|\leq1$  in $Q_{1,1}$ and
\[
\sup_{t\in(-1,0)}{\Tail}_{p-1,sp+1}(w(t,\cdot);1)\leq T_0.
\]
Then, for $t,\tau\in(-1,0]$, $x\in B_{3/4}$ and $0<|t-\tau|\leq1/8$,
\begin{equation}\label{eq:centered-time}
|w(t,x)-w(\tau,x)|\leq C(1+T_0^{p-1})
\begin{cases}
|t-\tau|,&0<\sigma<1,\\
|t-\tau|\big(1+\big|\log|t-\tau|\big|\big),&\sigma=1,\\
|t-\tau|^{1/\sigma},&1<\sigma<2.
\end{cases}
\end{equation}
Moreover, for any $0<r<R\leq 1$ the function
\[
F(t,x):=\int_{B_R^c}\frac{J_p(u(t,x)-u(t,y))}{|x-y|^{d+sp}}\dd y
\]
satisfies
\begin{equation}\label{eq:exterior-lipschitz}
[F(t,\cdot)]_{Lip(B_{r})}\leq C(1+T_0^{p-1}).
\end{equation}
Here $C$ depends only on $d,s,p$ as well as $r$ and $R$.
\end{lemma}

\begin{proof}
Estimate \eqref{eq:exterior-lipschitz} is essentially \cite[Lemma 6.9]{GJS}, so we skip its proof.

We focus on \eqref{eq:centered-time}. Fix a smooth nonnegative mollifier $\zeta$ supported in $B_1$, of integral one, and let $\zeta_r(z)=r^{-d}\zeta(z/r)$, with $0<r<1/8$. For fixed $x\in B_{3/4}$ set
\[
\psi(z):=\zeta_r(z-x)-\zeta_r(z),\qquad
M_r(t):=\int_{\R^d} u(t,z)\psi(z)\dd z.
\]
The spatial Lipschitz bound gives $|M_r(t)-w(t,x)|\leq Cr$. Testing the weak equation, we obtain
\begin{align*}
M_r'(t)={}&-\frac12\iint_{B_1\times B_1}
\frac{J_p(u(t,z)-u(t,y))(\psi(z)-\psi(y))}{|z-y|^{d+sp}}\dd y\dd z-\int_{B_1}F(t,z)\psi(z)\dd z.
\end{align*}
Since $\int\psi=0$, we may subtract $F(t,0)$ in the last term and use \eqref{eq:exterior-lipschitz}. Also
$\|\psi(\cdot+a)-\psi\|_{L^1}\leq C\min\{|a|/r,1\}$, so
\[
|M_r'|\leq C(1+T_0^{p-1})+C\int_0^2\rho^{-\sigma}\min\{\rho/r,1\}\dd\rho
\leq C(1+T_0^{p-1})
\begin{cases}
1,&\sigma<1,\\
1+|\log r|,&\sigma=1,\\
r^{1-\sigma},&\sigma>1.
\end{cases}
\]
Integrate between $t$ and $\tau$. For $\sigma<1$ let $r\to0$, for $\sigma=1$ take $r=|t-\tau|/16$ and for $\sigma>1$ take $r=|t-\tau|^{1/\sigma}/16$. 
\end{proof}

\section{First stage: iterating the reduction of the gradient} \label{s:iteration}

We obtain an improvement of flatness for the gradient by iterating Lemma~\ref{l:iteration} while the degenerate assumption (4) holds. Here we assume $u$ is a weak solution which is smooth in $x$.

\begin{lemma}[Iterated reduction of the gradient]\label{l:iterated-DG}
Fix $r,\mu,\tau\in(0,1/2)$. There exist constants $T_0\geq1$, $\delta,\eps_1\in(0,1/2)$, $K_0\in\mathbb N$ and $\rho_1,\rho_2\in(0,1/2)$, depending only on $r,\mu,\tau,d,s,p$, with the following property. 
Set
\begin{equation}\label{eq:rho1-def}
\rho_1:=\rho_2^\sigma(1-\delta)^{2-p},
\end{equation}
and
\begin{equation}\label{eq:un-def}
\lambda:=\rho_2(1-\delta),\qquad
u_n(t,x):=\lambda^{-n}u(\rho_1^nt,\rho_2^nx).
\end{equation}
Assume that
\begin{enumerate}
\item $u$ is a solution of \eqref{eq:fracplap} in $Q_{1,2^{K_0+1}}$;
\item $\|\nabla u\|_{L^\infty(Q_{1,2^k})}\leq(1-\delta)^{-k}$ for $k=0,\ldots,K_0$;
\item $\sup_{t\in(-1,0]}{\Tail}_{p-1,sp+1}(w(t,\cdot);1)\leq T_0$;
\item $\sup_{t\in(-1,0]}{\Tail}_{p-1,sp+1}(w(t,\cdot);2^{K_0-1})\leq\eps_1$.
\end{enumerate}
For $e\in S^{d-1}$, define
\[
\A_{r,n,e}(t):=\{x\in B_1:\nabla u_n(t,x)\notin B_r(e)\}.
\]
Let $N\geq1$ and assume that
\begin{equation}\label{eq:iterated-reservoir}
\int_{-1}^{-2\tau}|\A_{r,n,e}(t)|\dd t\geq\mu
\end{equation}
for every $e\in S^{d-1}$ and $n=0,\ldots,N-1$. Then, for $n=0,\ldots,N$,
\begin{align}
\|w_n\|_{L^\infty(Q_{1,1})}&\leq 1,\label{eq:un-linfty}\\
\|\nabla u_n\|_{L^\infty(Q_{1,2^k})}&\leq(1-\delta)^{-k},\qquad k=0,\ldots,K_0,\label{eq:un-gradient-growth}\\
\sup_{t\in(-1,0]}{\Tail}_{p-1,sp+1}(w_n(t,\cdot);1)&\leq T_0,\label{eq:un-sp-tail}\\
\sup_{t\in(-1,0]}{\Tail}_{p-1,sp+1}(w_n(t,\cdot);2^{K_0-1})&\leq\eps_1.\label{eq:un-sp1-tail}
\end{align}
Consequently,
\begin{equation}\label{eq:original-gradient-decay}
\|\nabla u\|_{L^\infty(Q_{\rho_1^n,\rho_2^n})}\leq(1-\delta)^n,
\qquad n=0,\ldots,N.
\end{equation}
\end{lemma}

\begin{proof}
Let
\[
T_0^{p-1}:=\max\left\{1+\frac{|S^{d-1}|}{\sigma},\frac{2^p|S^{d-1}|}{\sigma}\right\}
\]
 and choose $\delta=\min\{\delta_1,\delta_2\}$, $\eps_1$ and $K_0$ from Lemma~\ref{l:iteration}. Increase $K_0$ so that
\[
\frac{|S^{d-1}|}{\sigma(1-\delta)^{p-1}}2^{-\sigma(K_0-1)}\leq\frac{\eps_1^{p-1}}2.
\]
Then choose $\rho_2$ sufficiently small that
\[
\rho_1\leq\tau/16,\qquad \rho_2\leq2^{-K_0},\qquad
\frac{\rho_2^\sigma}{(1-\delta)^{p-1}}\leq\frac12,\qquad
\frac{\rho_2^\sigma T_0^{p-1}}{(1-\delta)^{p-1}}\leq\frac{\eps_1^{p-1}}2.
\]
These choices are possible for every $\sigma>0$.

We prove \eqref{eq:un-linfty}--\eqref{eq:un-sp1-tail} by induction. They hold for $n=0$. Suppose they hold at some $n<N$. Since $\rho_1=\rho_2^{sp}\lambda^{2-p}$, the function $u_{n+1}$ solves the same equation. Lemma~\ref{l:iteration}, applied to $u_n$ for every $e$, gives
\begin{equation}\label{eq:un-small-cylinder-gradient}
|\nabla u_n|\leq1-\delta\quad\text{in }Q_{\tau/16,1/2}.
\end{equation}
Thus $|\nabla u_{n+1}|\leq1$ in $Q_{1,1}$. For $k\geq1$, the inclusion $B_{2^k\rho_2}\subset B_{2^{k-1}}$ and the induction hypothesis give
\[
\|\nabla u_{n+1}\|_{L^\infty(Q_{1,2^k})}
\leq(1-\delta)^{-1}\|\nabla u_n\|_{L^\infty(Q_{1,2^{k-1}})}
\leq(1-\delta)^{-k}.
\]
This proves \eqref{eq:un-gradient-growth}. Since $w_{n+1}(t,0)=0$, it also proves \eqref{eq:un-linfty}.

Set $S_n(t):={\Tail}_{p-1,sp+1}(w_n(t,\cdot);1)^{p-1}$. Changing variables gives
\begin{equation}\label{eq:Sn-scaling}
S_{n+1}(t)=\frac{\rho_2^\sigma}{(1-\delta)^{p-1}}
\int_{B_{\rho_2}^c}\frac{|w_n(\rho_1t,z)|^{p-1}}{|z|^{d+sp+1}}\dd z.
\end{equation}
In $B_1$ we have $|w_n(t,z)|\leq|z|$. Splitting at radius one, we obtain
\begin{equation}\label{eq:Sn-recursion}
S_{n+1}(t)\leq\frac{\rho_2^\sigma}{(1-\delta)^{p-1}}S_n(\rho_1t)
+\frac{|S^{d-1}|(1-\rho_2^\sigma)}{\sigma(1-\delta)^{p-1}}
\leq T_0^{p-1}.
\end{equation}
For $1\leq R\leq\rho_2^{-1}$, the same split gives
\begin{equation}\label{eq:large-centered-tail}
{\Tail}_{p-1,sp+1}(w_{n+1}(t,\cdot);R)^{p-1}
\leq\frac{\rho_2^\sigma T_0^{p-1}}{(1-\delta)^{p-1}}
+\frac{|S^{d-1}|R^{-\sigma}}{\sigma(1-\delta)^{p-1}}.
\end{equation}
Taking $R=2^{K_0-1}$ proves \eqref{eq:un-sp1-tail} by our choices of $K_0$ and $\rho_2$. The induction is complete, and \eqref{eq:original-gradient-decay} follows by rescaling.
\end{proof}

We begin the next stage of this paper by considering the rescaling from Lemma~\ref{l:iterated-DG}. We collect its properties in the following corollary.

\begin{corollary}\label{c:final_first_stage}
Take $\tau=\mu/(2|B_1|)$ with $\mu<|B_1|$. Let $u$ satisfy the assumptions of Lemma~\ref{l:iterated-DG}. Let $\kmax<\infty$ be the first nonnegative integer for which \eqref{eq:iterated-reservoir} fails for some $e\in S^{d-1}$, and set
\[
\bar u(t,x):=\lambda^{-\kmax}u(\rho_1^{\kmax}t,\rho_2^{\kmax}x).
\]
Then $\bar u$ satisfies:
\begin{enumerate}
\item $\bar u$ solves \eqref{eq:fracplap} in $Q_{\rho_1^{-\kmax},2^{K_0+1}\rho_2^{-\kmax}}$;
\item $\|\bar u(t,x)-\bar u(t,0)\|_{L^\infty(Q_{1,1})}\leq 1$;
\item $\|\nabla\bar u\|_{L^\infty(Q_{1,2^k})}\leq(1-\delta)^{-k}$ for $k=0,\ldots,K_0+\kmax$;
\item $\sup_{t\in(-1,0]}{\Tail}_{p-1,sp+1}(\bar u(t,\cdot)-\bar u(t,0);1)\leq T_0$;
\item for some $e\in S^{d-1}$,
\[
|\{(t,x)\in Q_{1,1}:\nabla\bar u(t,x)\in B_r(e)\}|>|Q_{1,1}|-2\mu.
\]
\end{enumerate}
\end{corollary}

The proof of the corollary follows immediately from Lemma~\ref{l:iterated-DG}.

\section{Second stage: the Ishii-Lions method} \label{s:stage2}

In this section we deal with the second stage of the proof. As before, we assume $u$ is smooth in $x$, whereas time derivatives are to be understood in the viscosity sense. We aim at proving that, if $\nabla u$ is very close to a unit vector in most of the cylinder $Q_{1,1}$, then $\nabla u$ is close to that vector for every point in $Q_{1/2,1/2}$, resulting in some form of nondegeneracy for the linearized kernel $K_u$. The following lemma is the objective of this section. 

\begin{lemma} \label{l:stage2}
    Let $L\in(0,1/4)$ and $ T_0>0$.  There exist $\mu$ and $r$ sufficiently small, depending on  $L$, and $T_0$, such that the following statement holds.
    Assume that $u\in C(-1,0;L^{p-1}_{sp}(\R^d))$ is a solution of \eqref{eq:fracplap} in $Q_{1,2}$ with $p \in [2, \frac{2}{1-s})$, satisfying  $|\nabla u| \leq 1$ in $Q_{1,1}$, $\sup_{t \in(-1,0)}{\Tail}_{p-1,sp+1}(w(t,\cdot);1)\leq T_0$. 
    Assume further that, for some $e \in S^{d-1}$,
    \[ 
    |\{(t,x) \in Q_{1,1} :\nabla u(t,x)\in B_r(e)\}| \geq |Q_{1,1}|-2\mu.
    \]
     Then
     \[
     |\nabla u(t,x) - e| \leq L \quad \text{ for all } (t,x)\in  Q_{1/2,1/2}.
     \]
\end{lemma}

The next lemma turns the measure estimate into a pointwise one.

\begin{lemma} \label{l:eps_osc}
    Given any $\eps_0>0$, there are small values for $r>0$ and $\mu>0$, depending on \(\varepsilon_0,T_0,d,s,p\), so that if $u$ satisfies the assumptions of Lemma \ref{l:stage2} then 
\begin{align*}
 \underset{x \in B_{3/4}}{\osc}\,(u(t,x) -  e \cdot x)\leq \eps_0 \quad \text{for all }t \in (-3/4, 0].
\end{align*}
\end{lemma}

\begin{proof}
    Let $f(t):=|\{x \in B_1 \, : \, \nabla u (t,x)\notin B_r(e)\}|$. By assumption, we know $\int_{-1}^{0} f(t) \dd t \leq 2\mu$. Define $E:=\{ t \in (-3/4, 0] \, : \, f(t) \geq \sqrt{\mu}\}$. Then
    \[
    |E| \leq \frac{1}{\sqrt{\mu}}\int_E f(t) \dd t  \leq \frac{1}{\sqrt{\mu}}\int_{-1}^{0} f(t) \dd t \leq  2\sqrt{\mu},
    \]
  and for $t \in (-3/4,0) \setminus E$ we have \
  \[
  |\{x \in B_1 \, : \, \nabla u (t,x)\notin B_r(e)\}| \leq \sqrt{\mu}. 
  \]
  Combining this with Morrey's inequality (see for instance \cite[$\S$  5.6.2]{Evans})  for $t \in (-3/4,0) \setminus E$ we obtain
 \begin{equation} \label{eq:morrey}
     \osc_{x \in B_{1}} (u(t,\cdot )-e\cdot x) \leq C \| \nabla u(t,\cdot) - e\|_{L^{2d}(B_1)} \leq C \left( r + \mu^{\frac{1}{4d}}\right)
 \end{equation}
 where $C$ is a positive dimensional constant.

For $t\in E$ choose $t_0\in(-3/4,0)\setminus E$ with $|t-t_0|\leq3\sqrt\mu$. By Lemma~\ref{l:centered-time}, the difference $|w(t,x)-w(t_0,x)|$ tends to zero uniformly for $x\in B_{3/4}$ as $\mu\to0$, with a bound depending only on $T_0,d,s,p$. Spatial oscillations are unchanged by subtracting $u(t,0)$. Combining this with \eqref{eq:morrey}, and taking first $\mu$ and then $r$ sufficiently small, proves the lemma.

\end{proof}

Lemma \ref{l:stage2} follows from the next result together with Lemma \ref{l:eps_osc}. 
Here we employ an Ishii-Lions argument in the spirit of \cite[Section 6]{GJS} with the techniques introduced in \cite[Section 4]{JSU} for the parabolic setting to show that the smallness of oscillation in space at each time-slice is enough to ensure the gradient of $u$ is close to the direction $e$ pointwise.

\begin{lemma}\label{l:IL}
Let $0<L\leq 1/4$, and $T_0> 0$. There exists some $\eps_0>0$ small enough so that the following statement is true.
    Assume $u$ is a solution of \eqref{eq:fracplap} in $Q_{1,2}$ for $p \in [2, \frac{2}{1-s})$, $u\in C(-1,0;L^{p-1}_{sp}(\R^d))$,  $ |\nabla u| \leq 1 $  in $Q_{1,1}$ and $\sup_{t \in (-1,0]}{\Tail}_{p-1,sp+1}(w(t,\cdot);1) \leq T_0$.  Furthermore for some $e \in S^{d-1}$ we have 
    \begin{equation} \label{eq:osc_small}
        \underset{x \in B_{3/4}}{\osc}\,(u(t,x) -  e \cdot x)\leq \eps_0 \quad \text{for any }t  \in (-3/4, 0].
    \end{equation} 
Then, it holds that
\begin{equation*} 
    |\nabla u - e| \leq L \quad \text{in }Q_{1/2,1/2}.
\end{equation*}
Here $\eps_0$ depends only on $L$, $T_0$, $d$, $s$ and $p$.
\end{lemma}

The purpose of the rest of this section is to prove Lemma \ref{l:IL}. The idea is to apply the Ishii-Lions method to the function $v(t,x) = u(t,x) - e \cdot x$, which has a small oscillation in $B_{3/4}$ at each time by Lemma \ref{l:eps_osc}.

\subsection{Setting up the Ishii-Lions method}

Our aim is to prove that for $v(t,x)=u(t,x)-e\cdot x$,
\begin{equation}\label{thesis_IL}
\|\nabla v\|_{L^\infty(Q_{1/2,1/2})}\leq L.
\end{equation}
Fix $t_0\in(-1/2,0)$ and consider
\[
g(t,x,y):=v(t,x)-v(t,y)-L\omega(|x-y|)-\kappa\eps_0\psi(x)
-\kappa\eps_0(t_0-t)^{1+\gamma}
\]
on $[-3/4,t_0]\times\overline B_{3/4}\times\overline B_{3/4}$. We use the same strictly concave modulus as in the elliptic case,
\[
\omega(r)=r+\frac{r}{20\log(r/4)}.
\]
For $r\in(0,1/4]$,
\[
-\frac{C}{r\log^2r}\leq\omega''(r)\leq-\frac{c}{r\log^2r}.
\]
Let $\psi=\psi_0^m$, where $\psi_0$ is smooth, vanishes on $\overline B_{1/2}$, is positive outside it, and equals one outside $B_{9/16}$. Choose
\[
m>\max\{4,2/(2-\sigma)\},\qquad
\gamma>\max\{1,\sigma^{-1}-1\}.
\]
In particular,
\begin{equation}\label{eq:inequality_psi}
|\nabla\psi(x)|\leq C\psi(x)^{(m-1)/m}.
\end{equation}
We take $\kappa$ large depending on $m,\gamma$. By \eqref{eq:osc_small}, a positive maximum of $g$ then has $t>-3/4$ and $x\in B_{9/16}$. Moreover, $L\omega(|x-y|)\leq\eps_0$, so taking $\eps_0$ small forces $|x-y|<1/16$ and $y\in B_{5/8}$.

Suppose, for a contradiction, that the maximum is positive. For $\delta>0$, double the time variable and set
\begin{align*}
\Phi_\delta(t,\tau,x,y):={}&v(t,x)-v(\tau,y)-L\omega(|x-y|)-\kappa\eps_0\psi(x)-\kappa\eps_0(t_0-t)^{1+\gamma}-\eps_0\frac{(t-\tau)^2}{2\delta}.
\end{align*}
Let $(t_\delta,\tau_\delta,x_\delta,y_\delta)$ be a maximizer and write $a_\delta=x_\delta-y_\delta$. We have
\begin{equation}\label{eq:M_delta}
M_\delta:=\Phi_\delta(t_\delta,\tau_\delta,x_\delta,y_\delta)
\geq\max g=g(t,x,y)>0.
\end{equation}
The local uniform continuity of $u$ gives, along a subsequence,
\[
(t_\delta,\tau_\delta,x_\delta,y_\delta)\longrightarrow(t,t,x,y),
\qquad a:=x-y\ne0.
\]
The limiting point maximizes $g$. Thus the contact points are spatially interior and away from the initial time for small $\delta$.

Set
\[
\phi(x,y):=e\cdot(x-y)+L\omega(|x-y|)+\kappa\eps_0\psi(x).
\]
As in \cite[Section 4]{JSU}, use the tests

\begin{align*}
&w_{1}(t,x) := \begin{cases} \Lambda_1+\phi(x,y_\delta) + \kappa \eps_0 (t_0-t)^{1+\gamma} + \varepsilon_0\frac{(t - \tau_\delta)^2}{2\delta} & \text{if }x \in   B_{\eps_1(a_\delta)|a_\delta|}(x_\delta) \\ u(t,x) & \text{otherwise} \end{cases}\\
&w_{2}( \tau,y) := \begin{cases} \Lambda_2-\phi(x_\delta,y) -  \varepsilon_0\frac{(t_\delta - \tau)^2}{2\delta}  &\hspace{1in} \text{if } y \in B_{\eps_1(a_\delta)|a_\delta|}(y_\delta) \\ u(\tau,y) &\hspace{1in} \text{otherwise} \end{cases}
   \end{align*}
with the constants $\Lambda_1, \Lambda_2$ chosen to give equality at the contact points. We take $\eps_1(a)=c_1|\log|a||^{-q}$ and $q$ is fixed below.  The derivatives of the doubling time penalty cancel, giving
\begin{equation}\label{eq:master_ineq_IL1}
\dsp w_1(t_\delta,x_\delta)-\dsp w_2(\tau_\delta,y_\delta)
\leq\kappa\eps_0(1+\gamma)(t_0-t_\delta)^\gamma.
\end{equation}

We now let $\delta\to0$ in \eqref{eq:master_ineq_IL1}, before estimating the spatial integrals. Since $\eps_1(a_\delta)\leq1/2$ we have  $y_\delta\notin B_{\eps_1(a_\delta)|a_\delta|}(x_\delta)$ and thus $w_1(t,x_\delta)-w_1(t,x_\delta+z)$ is smooth in $z\in B_{\eps_1(a_\delta)|a_\delta|}$ and constant in time. 
Since $u\in C(-1,0;L^{p-1}_{sp}(\R^d))$, we can argue as in \cite[Lemma 3.3]{JSU} to take the limit $\delta\to 0$ in \eqref{eq:master_ineq_IL1} and we get
\begin{equation}\label{eq:master_ineq_IL}
\mathcal{I}:=\dsp w_1(t,x)-\dsp w_2(t,y)
\leq\kappa\eps_0(1+\gamma)(t_0-t)^\gamma.
\end{equation}

The following lemma bounds the right-hand side of this inequality.

\begin{lemma}\label{l:estimate_time}
At the limiting contact point,
\[
\mathcal I\leq C|a|^{\gamma/(1+\gamma)},
\]
where $C$ depends on $\kappa,\gamma$ and universal constants.
\end{lemma}

\begin{proof}
The positive maximum \eqref{eq:M_delta} and the spatial Lipschitz bound give
\[
\kappa\eps_0(t_0-t)^{1+\gamma}<v(t,x)-v(t,y)\leq2|a|.
\]
Combining this with  \eqref{eq:master_ineq_IL} proves the lemma.
\end{proof}

For $D\subset\R^d$, introduce the notation
\[
\mathcal L[D]w(t,x):=\int_DJ_p(w(t,x)-w(t,x+z))|z|^{-d-sp}\dd z.
\]
Let
\[
\mathcal K(a):=\{z\in B_{|a|/2}:|a\cdot z|\geq\sqrt{1-\eta(a)^2}|a|\,|z|\},
\qquad \eta(a)=c_0|\log|a||^{-1},
\]
and
\[
\mathcal D_1=B_{\eps_1(a)|a|}\cap\mathcal K^c,\qquad
\mathcal D_2=B_{1/16}\setminus(\mathcal D_1\cup\mathcal K).
\]
As before, we split the operator difference into
\begin{equation}\label{contrad2}
\begin{aligned}
\mathcal I={}&\underbrace{\mathcal L[\mathcal K]w_1(t,x)-\mathcal L[\mathcal K]w_2(t,y)}_{I_1}
+\underbrace{\mathcal L[\mathcal D_1]w_1(t,x)-\mathcal L[\mathcal D_1]w_2(t,y)}_{I_2}\\
&+\underbrace{\mathcal L[\mathcal D_2]w_1(t,x)-\mathcal L[\mathcal D_2]w_2(t,y)}_{I_3}
+\underbrace{\mathcal L[B_{1/16}^c]w_1(t,x)-\mathcal L[B_{1/16}^c]w_2(t,y)}_{I_4}.
\end{aligned}
\end{equation}

\begin{lemma}[Estimate of $I_1,\,I_2$ and $I_3$]\label{l:I123}
For every $L\in(0,1/4]$ and $\eps_0$ sufficiently small,
\begin{equation}\label{eq:bound123}
\begin{aligned}
I_1&\geq cL\eta^{d+p-1}|a|^{p(1-s)-1},\\
I_2&\geq-CL\eps_1^{p(1-s)}|a|^{p(1-s)-1},\\
I_3&\geq-C|a|^{p(1-s)/2}.
\end{aligned}
\end{equation}
The constants depend on $d,s,p$ and the chosen exponent $q$, but not on $a$.
\end{lemma}

\begin{proof}
The bounds follow as in \cite[Lemmas 6.6--6.8]{GJS}, so we only sketch it. For $I_1$, the cone contains a positive dimensional fraction of directions where $|\nabla_x\phi\cdot z|\geq\eta|z|/4$, since $|\nabla_x\phi|\geq1/2$. The comparison with the spatial tests then gives the first two bounds.

For $I_3$, use the simultaneous shift $(x,y)\mapsto(x+z,y+z)$. Positivity of $g$ gives $\psi(x)\leq C|a|/\eps_0$, so \eqref{eq:inequality_psi} yields $|\nabla\psi(x)|\leq C(|a|/\eps_0)^{(m-1)/m}$. The same split at $|z|=|a|^{1/2}$ as in \cite[Lemma 6.8]{GJS} gives
\[
I_3\geq-C\left(|a|^{1-1/m}\int_{\eps_1|a|}^{|a|^{1/2}}r^{-\sigma}\dd r
+|a|^{1-\sigma/2}\right).
\]
The first term is bounded by
\[
C\begin{cases}
|a|^{(3-\sigma)/2-1/m},&\sigma<1,\\
|a|^{1-1/m}|\log|a||,&\sigma=1,\\
\eps_1^{1-\sigma}|a|^{2-\sigma-1/m},&\sigma>1.
\end{cases}
\]
Our choice of $m$ bounds all three by $C|a|^{1-\sigma/2}$, proving the last assertion.
\end{proof}

Take $q>(d+p-1)/(2-\sigma)$ to absorb $I_2$ into $I_1$. Also $-I_3\ll I_1$ as $|a|\to0$, since $\sigma>0$.

The following result is a direct consequence of \cite[Lemma 6.9]{GJS}.
\begin{lemma}\label{l:I4}
We have $|I_4|\leq C|a|$, where $C$ depends only on $T_0,d,s,p$.
\end{lemma}

\subsection{Conclusion of the proof of Lemma~\ref{l:IL}}

\begin{customproof}{Lemma~\ref{l:IL}}
Combining Lemmas~\ref{l:estimate_time}, \ref{l:I123} and \ref{l:I4}, and dividing by $|a|^{1-\sigma}$, we obtain
\begin{equation}\label{eq:ineq_final_IL2}
cL\eta^{d+p-1}\leq C|a|^\sigma+C|a|^{\sigma-1/(1+\gamma)}.
\end{equation}
Both powers are positive by the choice of $\gamma$. Since $L\omega(|a|)\leq\eps_0$, taking $\eps_0$ sufficiently small gives a contradiction. Thus $g\leq0$. At $t=t_0$ and $x,y\in B_{1/2}$ the penalties vanish. Interchanging $x,y$ proves \eqref{thesis_IL}. The constants are independent of $t_0<0$, so the estimate also holds at $t_0=0$.
\end{customproof}

We may combine the results of Section \ref{s:iteration} with Section \ref{s:stage2} to obtain the following.

\begin{corollary}\label{c:stages1and2}

Fix $T_0$ as in the proof of Lemma~\ref{l:iterated-DG}. 
Choose $r,\mu\in(0,1/2)$ sufficiently small as in Lemma~\ref{l:stage2} with $L=1/8$, and set $\tau=\mu/(2|B_1|)\in(0,1/2)$. There exist $\delta, \eps_1,\rho_1,\rho_2 \in (0, 1/2)$ and $K_0 \in \mathbb{N}$, depending only on $r,\,\mu,\,\tau, \, d,\, s,\, p$, such that if the following conditions hold
\begin{enumerate}
    \item $u$ is a solution of \eqref{eq:fracplap} in $Q_{1,2^{K_0+1}}$,
   \item $\|\nabla u\|_{L^{\infty}(Q_{1,2^{n}})} \leq (1-\delta)^{-n}$ for $n=0,1,\dots,K_0$,
    \item $\sup_{t \in (-1,0]} {\Tail}_{p-1,sp+1}(u(t,\cdot)-u(t,0);1) \leq T_0$,
\item $\sup_{t\in(-1,0]}{\Tail}_{p-1,sp+1}(u(t,\cdot)-u(t,0);2^{K_0-1})\leq\eps_1$.

    \end{enumerate}
then one of the following must hold.
\begin{itemize}
\item There exists a nonnegative integer $\kmax$ such that 
    \[
    \| \nabla u \|_{L^{\infty}(Q_{\rho_1^n,\rho_2^n})} \leq (1-\delta)^n \quad \text{for } n=0,\dots, \kmax,
    \]
    and there exists $e \in S^{d-1}$ such that
\begin{equation}\label{eq:kmax_finite}
    |\nabla u - (1-\delta)^{\kmax} e| < \frac{(1-\delta)^{\kmax}}{4} \quad \text{in } Q_{\rho_1^\kmax/2,\rho_2^\kmax/2}, 
\end{equation}
    \item  For $\alpha_0>0$ so that $\rho_2^{\alpha_0} = (1-\delta)$ and $C_0 = (1-\delta)^{-1}$, we have
    \begin{equation} \label{eq:kmax_infinite}
        |\nabla u(t,x)| \leq C_0 (|t|^{\frac{1}{\sigma}}+ |x|)^{\alpha_0} \quad \text{ for } (t,x) \in  Q_{1,1}.
    \end{equation}
 \end{itemize}
\end{corollary}

\begin{proof}

Choose the constants as in Lemma~\ref{l:iterated-DG}. Its four assumptions are precisely the hypotheses above.

    Let $\kmax$ be the first nonnegative integer for which \eqref{eq:iterated-reservoir} fails for some direction $e\in S^{d-1}$. The proof is a dichotomy. First, if $\kmax<\infty$, then we start by applying Corollary \ref{c:final_first_stage} and check that the rescaling $\bar u$ satisfies the assumptions of Lemma \ref{l:stage2} which in turn implies \eqref{eq:kmax_finite} after rescaling back to $u$.

The second alternative is $\kmax=\infty$, that is, when \eqref{eq:iterated-reservoir} holds for every $n$ and $e$. In this case, we recall 
\[
\rho_1
:=
\frac{\rho_2^\sigma}{(1-\delta)^{p-2}}.
\]
Let $(t,x)\in Q_{1,1}\setminus\{(0,0)\}$ and choose an integer $n\geq0$ such that $(t,x)\in Q_{\rho_1^n,\rho_2^n}\setminus Q_{\rho_1^{n+1},\rho_2^{n+1}}$. Then either $|t|\geq\rho_1^{n+1}$ or $|x|\geq\rho_2^{n+1}$. Since $\rho_1\geq\rho_2^\sigma$, in both cases $|t|^{1/\sigma}+|x|\geq\rho_2^{n+1}$, and hence
\[
    (1-\delta)^n=C_0\rho_2^{(n+1)\alpha_0}\leq C_0\left(|t|^{1/\sigma}+|x|\right)^{\alpha_0}.
\]
Combining this with \eqref{eq:original-gradient-decay} gives \eqref{eq:kmax_infinite}. At $(0,0)$, the same conclusion follows by letting $n\to\infty$ in \eqref{eq:original-gradient-decay}.

\end{proof}

\section{Third stage: the nondegenerate case} \label{s:stage3}

To conclude the third stage of the proof, we analyze the linearized equation establishing a H\"older modulus of continuity for $|\nabla u|$ jointly in space and time. The nondegeneracy of the gradient obtained in Section \ref{s:stage2} allows us to work in a regime where the linearized kernel is mildly uniformly elliptic. This implies coercivity and opens the way for regularity results using more standard techniques for parabolic integro-differential equations. Here, we assume $u$ is a weak solution of \eqref{eq:fracplap} which is smooth in $x$. The main result in this section is the following.

\begin{lemma} \label{l:stage3}

Let $u$ be a solution of \eqref{eq:fracplap} in $Q_{2,2}$ with $p\in[2,2/(1-s))$,
\begin{equation} \label{eq:hp_stage3}
    \|\nabla u\|_{L^\infty(Q_{1,1})}\leq1,\qquad
\sup_{t\in(-1,0]}{\Tail}_{p-1,sp+1}(u(t,\cdot)-u(t,0);1)\leq T_0.
\end{equation}

      Suppose further that there is a vector $e \in S^{d-1}$ such that
      \[
      |\nabla u(t,x) -e| < 1/4 \quad \text{ for } (t,x) \in Q_{1,1}.
      \]
    Then there exist positive constants $\alpha_1$ depending on $d$, $s$ and $p$, and $C$ depending on $d$, $s$, $p$ and $T_0$ such that $\nabla u \in C^{\alpha_1}(Q_{1/16,1/16})$ with the estimate 
    \[
    \|\nabla u\|_{C^{\alpha_1}(Q_{1/16,1/16})} \leq C.
    \]
\end{lemma}

We recall the following results for nonlocal operators in divergence form in the literature, restated in a convenient way for the parabolic setting. The first one provides coercivity for our linearized kernel $K_u(t;\cdot, \cdot)$ under the nondegeneracy condition at each fixed time slice.  This result was obtained in \cite{chaker2020coercivity}.

\begin{theorem} \label{Thm:ChakerLuis}
    Assume there exist $\mu >0$ and $\lambda >0$ such that for every ball $B \subset \R^d$ and $x \in B$ the kernel $K$ satisfies the following hypothesis
     \begin{equation} \label{AssumptionA1}
        |\{ y \in B \, : \, K(x,y) \geq \lambda |x-y|^{-d-2\gamma}\}| \geq \mu |B|,
    \end{equation}
    for some $0<\gamma<1$ and uniformly in $t$.
    Then there exists $c$ depending only on $\mu$ and $d$ such that 
\begin{equation*}
        \int_{B_2} \int_{B_2}  K(x,y) (v(t, x)-v(t, y))^2 \dd y \dd x \geq c \lambda [ v(t, \cdot) ]_{W^{\gamma,2}(B_1)}^2.
    \end{equation*}
\end{theorem}

The next result follows mostly from \cite{imbert2019weak}. See also \cite{FelsingerKassmann,Kassman-Schwab,kassmannduke2024}. For more details we refer to the discussion after \cite[Theorem 7.3]{GJS}.

\begin{theorem} \label{Thm:Kassman&co}
    Let $v(t,x)$ be a bounded solution of 
      \[
      \partial_t v(t,x) + \int_{\R^d} K(t;x,y)\big( v(t,x) - v(t,y) \big) \dd y  = f(t,x) \quad \text{ for } (t,x)\in Q_{1,1},
      \]
    where $f$ is a bounded function in $Q_{1,1}$.
    Assume that there are $\gamma>0, \, \Lambda >1$ such that for every $x_0 \in B_1$ and uniformly in $t$, the following assumptions on the kernel are satisfied:
\begin{align*}
  \textit{(A1)} &\quad K(t;x,y)   =  K(t;y,x),  \\
  \textit{(A2)} &\quad \rho^{-2} \int_{B_\rho(x_0)} |x_0-y|^2 K(t;x_0,y) \dd y + \int_{B_\rho^{c}(x_0)} K(t;x_0,y) \dd y \leq \Lambda \rho^{-2\gamma}, \\
    \textit{(A3)} &\quad \text{Whenever $B_\rho(x_0) \subset B_1$ and $v(t, \cdot) \in H^{\gamma}(B_{\rho}(x_0))$, }  \\
    &\qquad \qquad \int_{B_{\rho}(x_0)} \int_{B_{\rho}(x_0)} K(t;x,y) \left(v(t,x)-v(t,y)\right)^2 \dd y \dd x   \geq \Lambda^{-1} [ v(t, \cdot) ]_{W^{\gamma,2}(B_{\rho/2}(x_0))}^2.
\end{align*}
    Then there exist positive constants $\alpha_1$ and $C$ depending on $\Lambda, \, \gamma$ and $d$ such that
    \[
    [v]_{C^{\alpha_1}(Q_{1/2,1/2}) } \leq C \left( \|v \|_{L^{\infty}((-1,0]\times\R^{d})} +  \|f\|_{L^{\infty}(Q_{1,1})} \right).
    \]
\end{theorem}

We will see that when $\nabla u$ stays sufficiently close to some unit vector $e$, the linearized kernel $K_u$ is nondegenerate in a cone of directions. This allows us to verify \eqref{AssumptionA1} at each time slice and get the coercivity bound through Theorem \ref{Thm:ChakerLuis}. As a consequence, we can apply Theorem \ref{Thm:Kassman&co} to get a H\"older modulus of continuity jointly in space and time. We detail this in the following proof.

\begin{proof}[Proof of Lemma~\ref{l:stage3}]
Fix $\omega\in S^{d-1}$ and define
\[
v_\omega(t,x):=\eta(8x)\,(\omega\cdot\nabla u(t,x)).
\]
Then
\[
v_\omega=\omega\cdot\nabla u \quad\text{in }B_{3/16},
\qquad
\supp v_\omega(t,\cdot)\subset B_{7/32}.
\]
Moreover,
\[
\sup_{t\in(-1,0]}
{\Tail}_{p-1,sp+1}(w(t,\cdot);1/8)^{p-1}
\leq C\bigl(1+T_0^{p-1}\bigr).
\]
Therefore, Lemma~\ref{l:eqforv_e_scaled}, applied with $R=1/8$, gives
\begin{equation}\label{eq:ve_RHS}
\left|
\partial_t v_\omega(t,x)+\LL_u v_\omega(t,x)
\right|
\leq
C\bigl(1+T_0^{p-1}\bigr)
\end{equation}
for every $(t,x)\in Q_{1/8,1/8}$.

We next record the nondegeneracy of the linearized kernel. If
$x,y\in B_{1/2}$ and $|(y-x)\cdot e|>\frac12|x-y|$
then the segment joining $x$ and $y$ is contained in $B_{1/2}$.
Writing $y=x+\rho\nu$, $\rho=|x-y|$ and $\nu=\frac{y-x}{|y-x|}$, we obtain
\[
\left|\nabla u(t,x+s\nu)\cdot\nu\right|
\geq |\nu\cdot e|
      -|\nabla u(t,x+s\nu)-e|
>\frac14
\]
for every $s\in[0,\rho]$. The sign of
$\nabla u(t,x+s\nu)\cdot\nu$ is constant along the segment, and hence
\[
|u(t,y)-u(t,x)|
=
\left|
\int_0^\rho \nabla u(t,x+s\nu)\cdot\nu\,\dd s
\right|
\geq \frac{\rho}{4}.
\]
Consequently,
\begin{equation}\label{eq:K_u_A3}
K_u(t;x,y)
\geq
c|x-y|^{-d-\sigma}.
\end{equation}
For the same $c$ as above, let
\[
\widetilde K(z):=c|z|^{-d-\sigma}.
\]
Define the symmetric kernel
\[
K(t;x,y):=
\begin{cases}
K_u(t;x,y),&x,y\in B_{1/2},\\[1mm]
\widetilde K(y-x),&\text{otherwise}.
\end{cases}
\]
For $x\in B_{1/8}$ and $y\in B_{1/2}^c$, we have
$v_\omega(t,y)=0$. Hence, using \eqref{eq:ve_RHS},
\begin{align*}
&\left|
\partial_t v_\omega(t,x)
+\int_{\R^d}
K(t;x,y)\bigl(v_\omega(t,x)-v_\omega(t,y)\bigr)\,\dd y
\right|
\\
&\quad\leq
\left|
\partial_t v_\omega(t,x)+\LL_u v_\omega(t,x)
\right|
+
|v_\omega(t,x)|
\int_{B_{1/2}^c}
\left(K_u(t;x,y)+\widetilde K(y-x)\right)\,\dd y .
\end{align*}
Using \eqref{eq:hp_stage3} and Lemma~\ref{l:Tail_inequalities}, we obtain that for every $(t,x)\in Q_{1/8,1/8}$,
\[
\int_{B_{1/2}^c}K_u(t;x,y)\,\dd y
\leq
C\bigl(1+T_0^{p-1}\bigr),
\]
while
\[
\int_{B_{1/2}^c}\widetilde K(y-x)\,\dd y\leq C.
\]
Since $\|v_\omega\|_{L^\infty((-1/8,0]\times\R^d)}\leq1$, it follows that for $(t,x) \in Q_{1/8,1/8}$
\begin{equation}\label{eq:modified-kernel-equation}
\left|
\partial_t v_\omega(t,x)
+\int_{\R^d}
K(t;x,y)\bigl(v_\omega(t,x)-v_\omega(t,y)\bigr)\,\dd y
\right|
\leq
C\bigl(1+T_0^{p-1}\bigr).
\end{equation}

We now verify the assumptions of Theorem~\ref{Thm:Kassman&co}.
Symmetry follows directly from the definition of $K$. Moreover, if
$x,y\in B_{1/2}$, then
\[
|u(t,x)-u(t,y)|\leq |x-y|,
\]
and therefore
\[
K_u(t;x,y)\leq (p-1)|x-y|^{-d-\sigma}.
\]
Together with the definition of $\widetilde K$, this gives the global bound
\begin{equation*}
K(t;x,y)\leq C|x-y|^{-d-\sigma}.
\end{equation*}
Consequently, for every $\rho>0$,
\[
\rho^{-2}
\int_{B_\rho(x_0)}
|x_0-y|^2K(t;x_0,y)\,\dd y
+
\int_{B_\rho^c(x_0)}
K(t;x_0,y)\,\dd y
\leq C\rho^{-\sigma}.
\]

It remains to verify the coercivity assumption. The double cone
\[
\left\{y\in\R^d:
|(y-x)\cdot e|>\frac12|x-y|\right\}
\]
occupies a fixed dimensional proportion of every ball containing $x$.
If $x\notin B_{1/2}$, then
\[
K(t;x,y)=\widetilde K(y-x)
\]
for every $y$. If $x\in B_{1/2}$, the lower bound
\[
K(t;x,y)\geq c|x-y|^{-d-\sigma}
\]
holds whenever either $y\notin B_{1/2}$ or
\[
y\in B_{1/2}
\quad\text{and}\quad
|(y-x)\cdot e|>\frac12|x-y|,
\]
by \eqref{eq:K_u_A3}. 
It follows that there exists  a dimensional constant $c_1>0$
such that, for every ball $B\subset\R^d$ and every $x\in B$,
\[
\left|
\left\{
y\in B:
K(t;x,y)\geq c|x-y|^{-d-\sigma}
\right\}
\right|
\geq c_1|B|.
\]
Theorem~\ref{Thm:ChakerLuis}, applied at each fixed time with
$\gamma=\sigma/2$, therefore yields Assumption \textit{(A3)}, with
constants depending only on $d,s,p$.

We may now apply Theorem~\ref{Thm:Kassman&co} after rescaling
$Q_{1/8,1/8}$ to $Q_{1,1}$. Since the rescaled right-hand side
is bounded by \eqref{eq:modified-kernel-equation}, we obtain
\[
[v_\omega]_{C^{\alpha_1}
 (Q_{1/16,1/16})}
\leq
C\left(
\|v_\omega\|_{L^\infty((-1/8,0]\times\R^d)}
+1+T_0^{p-1}
\right)
\leq C,
\]
where $\alpha_1>0$ depends only on $d,s,p$, and
$C$ depends only on $d,s,p,T_0$.

Finally, $\eta(8x)=1$ on $B_{3/16}$, and therefore
\[
v_\omega=\omega\cdot\nabla u
\quad\text{in }Q_{1/16,1/16}.
\]
Since $\omega \in S^{d-1}$ is arbitrary, we get
\[
\|\nabla u\|_{C^{\alpha_1}
(Q_{1/16,1/16})}
\leq C,
\]
as claimed.
\end{proof}

\section{Proof of the main theorem}\label{s:proof_main}

We now prove Theorem~\ref{t:main}. The argument is the parabolic analogue of the proof in the elliptic case, but we must keep track of the geometry of the cylinders used in the first stage.

\begin{proof}[Proof of Theorem~\ref{t:main}]

Fix \(L=1/8\). We first fix  \(T_0\) as in the proof of Lemma~\ref{l:iterated-DG}; these constants depend only on \(d,s,p\). We then choose \(r\) and \(\mu\) as in Lemma~\ref{l:stage2} and, decreasing \(\mu\) if necessary, set
\[
\tau:=\frac{\mu}{2|B_1|}\in(0,1/2).
\]
Finally, fix \(\delta,\eps_1,\rho_1,\rho_2\) and \(K_0\) as in Corollary~\ref{c:stages1and2}. Recall that
\[
\lambda=\rho_2(1-\delta),
\qquad
\rho_1=\rho_2^\sigma(1-\delta)^{2-p}.
\]

We first normalize the size of $u$ by replacing it with $M^{-1}u(M^{2-p}t,x)$, according to the scaling properties in Section \ref{ss:scaling}. Since $M\geq1$, the rescaled equation is valid in $Q_{2,2}$ and its local height and tail there are bounded by one. Moreover, by Theorem~\ref{t:JSU-Lip}, choosing $M$ larger if necessary we can further assume
\[
\|\nabla u\|_{L^\infty(Q_{1,1})}\leq1,
\qquad
\sup_{t\in(-1,0]}{\Tail}_{p-1,sp+1}(w(t,\cdot);1)^{p-1}\leq C,
\]
where $C$ depends only on $d,s,p$. 

Choose $\theta>0$ small and set
\[
u_0(t,x):=\theta^{-1}u(\theta^\sigma t,\theta x).
\]
Taking $\theta\leq2^{-K_0-1}$, we have $|\nabla u_0|\leq1$ in $Q_{1,2^{K_0}}$, $|u_0(t,x)-u_0(t,0)|\leq|x|$ in $B_1$. For $R\geq1$ with $\theta \leq R^{-1}$, a change of variables and splitting the integral at radius one, we get
\[
{\Tail}_{p-1,sp+1}(u_0(t,\cdot)-u_0(t,0);R)^{p-1}
\leq\frac{|S^{d-1}|}{\sigma}R^{-\sigma}+C\theta^\sigma.
\]
By the choice of $K_0$ in Lemma~\ref{l:iterated-DG}, the first term at $R=2^{K_0-1}$ is at most $\eps_1^{p-1}/2$. Choose $\theta$ small enough that the tail at radius one is at most $T_0$ and the tail at $2^{K_0-1}$ is at most $\eps_1$. All choices depend only on $d,s,p$. Thus $u_0$ satisfies the hypotheses of Corollary~\ref{c:stages1and2}. For notational simplicity, we replace $u_0$ by $u$.

We now distinguish the two alternatives in
Corollary~\ref{c:stages1and2}. If the second alternative holds, then
\[
|\nabla u(t,x)|
\leq
C_0\bigl(|t|^{1/\sigma}+|x|\bigr)^{\alpha_0}\leq C\bigl(|t|^{\alpha_0/\sigma}+|x|^{\alpha_0}\bigr)
\qquad\text{in }Q_{1,1},
\]
where $\alpha_0$ is given by $\rho_2^{\alpha_0}=1-\delta$.
In particular, \(\nabla u(0,0)=0\).

Suppose instead that the first alternative holds. Then there exists a finite \(\kmax\) such that
\[
\|\nabla u\|_{L^\infty(Q_{\rho_1^{n},\;
  \rho_2^{n}})}
\leq(1-\delta)^n,
\qquad n=0,\ldots,\kmax,
\]
and there exists \(e\in S^{d-1}\) such that
\[
\left|
\nabla u-(1-\delta)^{\kmax}e
\right|
<
\frac{(1-\delta)^{\kmax}}4
\quad\text{in }
Q_{\rho_1^{\kmax}/2,\rho_2^{\kmax}/2}.
\]

Choose a fixed $a\leq\min\{1/2,2^{-1/\sigma}\}$ and define
\[
\bar u(t,x):=a^{-1}\lambda^{-\kmax}
u(a^\sigma\rho_1^{\kmax}t,a\rho_2^{\kmax}x).
\]
The scaling relation shows that $\bar u$ solves the same equation in $Q_{2,2}$. Moreover,
\[
\nabla\bar u(t,x)=(1-\delta)^{-\kmax}
\nabla u(a^\sigma\rho_1^{\kmax}t,a\rho_2^{\kmax}x).
\]
Since $a\leq1/2$ and $a^\sigma\leq1/2$, we have
\[
\|\nabla\bar u\|_{L^\infty(Q_{1,1})}\leq1,
\qquad |\nabla\bar u-e|<1/4\quad\text{in }Q_{1,1}.
\]
Also $|\bar u(t,x)-\bar u(t,0)|\leq1$ there. Splitting the centered tail at radius one gives
\[
\sup_{t\in(-1,0]}
{\Tail}_{p-1,sp+1}(\bar u(t,\cdot)-\bar u(t,0);1)^{p-1}
\leq\frac{|S^{d-1}|}{\sigma}+a^\sigma T_0^{p-1}.
\]

Thus \(\bar u\) satisfies all the hypotheses of
Lemma~\ref{l:stage3}. Hence there exist
\(\alpha_1>0\) and \(C>0\), depending only on \(d,s,p\), such that
\[
[\nabla\bar u]_
 {C^{\alpha_1}(Q_{1/16,1/16})}
\leq C.
\]

Fix a positive integer \(m_\ast\), depending only on \(d,s,p\), such that
\[
\rho_1^{m_\ast}\leq a^\sigma/16,\qquad\rho_2^{m_\ast}\leq a/16.
\]
Then
\[
Q_{\rho_1^{\kmax+m_\ast},\;
  \rho_2^{\kmax+m_\ast}}
\subset
{Q_{a^\sigma\rho_1^{\kmax}/16,\;a\rho_2^{\kmax}/16}}.
\]
Rescaling the preceding estimate back to \(u\), we obtain, for
\(z=(t,x)\) and \(z'=(\tau,y)\) in
\(Q_{\rho_1^{\kmax+m_\ast},\;
  \rho_2^{\kmax+m_\ast}}\),
\[
\begin{aligned}
|\nabla u(z)-\nabla u(z')|
&\leq
C(1-\delta)^{\kmax}
\left[
\left(
\frac{|t-\tau|}{\rho_1^{\kmax}}
\right)^{\alpha_1}
+
\left(
\frac{|x-y|}{\rho_2^{\kmax}}
\right)^{\alpha_1}
\right].
\end{aligned}
\]

Choose
\[
0<\alpha\leq \min\{\alpha_0,\alpha_0/\sigma,\alpha_1\}.
\]
Since \(z,z'\in Q_{\rho_1^{\kmax+m_\ast},\;
  \rho_2^{\kmax+m_\ast}}\), the two quotients in the preceding estimate are bounded by fixed constants. Hence
\[
\begin{aligned}
|\nabla u(z)-\nabla u(z')|
&\leq
C(1-\delta)^{\kmax}
\left[
\left(
\frac{|t-\tau|}{\rho_1^{\kmax}}
\right)^\alpha+
\left(
\frac{|x-y|}{\rho_2^{\kmax}}
\right)^{\alpha}
\right].
\end{aligned}
\]
Furthermore,
\[
(1-\delta)^{\kmax}
=
\rho_2^{\alpha_0\kmax}
\leq
\rho_2^{\sigma\alpha\kmax}
\leq
\rho_1^{\alpha\kmax}.
\]

Also $(1-\delta)^{\kmax}=\rho_2^{\alpha_0\kmax}\leq\rho_2^{\alpha\kmax}$. It follows that
\[
|\nabla u(z)-\nabla u(z')|\leq C\bigl(|t-\tau|^\alpha+|x-y|^\alpha\bigr)
\]
for \(z,z'\in Q_{\rho_1^{\kmax+m_\ast},\;
  \rho_2^{\kmax+m_\ast}}\).

It remains to treat points outside
\(Q_{\rho_1^{\kmax+m_\ast},\;
  \rho_2^{\kmax+m_\ast}}\). Let
\(z=(t,x)\in Q_{1,1}\setminus Q_{\rho_1^{\kmax+m_\ast},\;
  \rho_2^{\kmax+m_\ast}}\), and choose
\(k\in\{0,\ldots,\kmax+m_\ast-1\}\) such that
\[
z\in Q_{\rho_1^{k},\;
  \rho_2^{k}}\setminus Q_{\rho_1^{k+1},\;
  \rho_2^{k+1}}.
\]
If \(k\leq\kmax\), the first-stage estimate and the finite alternative give
\[
|\nabla u(z)-\nabla u(0,0)|
\leq C(1-\delta)^k.
\]
Since \(z\notin Q_{\rho_1^{k+1},\;
  \rho_2^{k+1}}\), either
\[
|x|\geq\rho_2^{k+1}
\qquad\text{or}\qquad
|t|\geq\rho_1^{k+1}.
\]
The choice of \(\alpha\) gives $1-\delta\leq\min\{\rho_2^\alpha,\rho_1^\alpha\}$,
and therefore
\[
|\nabla u(z)-\nabla u(0,0)|
\leq
C\bigl(|x|^\alpha+|t|^\alpha\bigr).
\]

If \(\kmax<k<\kmax+m_\ast\), then
\(z\in Q_{\rho_1^{\kmax},\;
  \rho_2^{\kmax}}\), and hence
\[
|\nabla u(z)-\nabla u(0,0)|
\leq C(1-\delta)^{\kmax}.
\]
Since \(m_\ast\) is fixed, the same argument gives
\[
\begin{aligned}
|\nabla u(z)-\nabla u(0,0)|
&\leq
C(1-\delta)^{-m_\ast}
(1-\delta)^{\kmax+m_\ast}\leq
C\bigl(|x|^\alpha+|t|^\alpha\bigr).
\end{aligned}
\]
Thus the finite alternative also gives
\[
|\nabla u(t,x)-\nabla u(0,0)|
\leq
C\bigl(|x|^\alpha+|t|^\alpha\bigr)
\qquad\text{in }Q_{1,1}.
\]

In the infinite alternative, the choice
\(\alpha\leq\min\{\alpha_0,\alpha_0/\sigma\}\) gives
\[
\begin{aligned}
|\nabla u(t,x)-\nabla u(0,0)|
&=
|\nabla u(t,x)|\leq
C\bigl(|t|^{1/\sigma}+|x|\bigr)^{\alpha_0}\leq
C\bigl(|t|^\alpha+|x|^\alpha\bigr).
\end{aligned}
\]
We have therefore proved the same pointwise estimate at the origin in both alternatives.

Undoing the preliminary rescaling only changes the constant by a factor depending on \(d,s,p\). The argument can be centered at any
\(z_0=(t_0,x_0)\in Q_{1,1}\). For this one applies the preceding construction to
\[
u(t_0+t,x_0+x)-u(t_0,x_0).
\]

Finally, let
\((t,x),(\tau,y) \in Q_{1,1}\), and assume without loss of generality that \(\tau\leq t\). If
\[
|t-\tau|^{1/\sigma}+|x-y|
\]
is smaller than the fixed radius used above, apply the pointwise estimate centered at \((t,x)\) to \((\tau,y)\). If it is larger, the same estimate follows from the spatial Lipschitz bound after enlarging \(C\). Consequently, under the normalized assumptions,
\[
|\nabla u(t,x)-\nabla u(\tau,y)|
\leq
C\left(
|t-\tau|^\alpha+|x-y|^\alpha
\right).
\]

Undoing the initial  normalization, including the time rescaling dictated by the homogeneity of the equation, gives
\[
|\nabla u(t,x)-\nabla u(\tau,y)|
\leq
CM\left( M^{(p-2)\alpha}|t-\tau|^\alpha+
|x-y|^\alpha
\right).
\]
This proves the theorem.
\end{proof}

\section{Time regularity}

In this section we prove Theorem~\ref{t:time_reg}.
Using the proof of Lemma \ref{l:centered-time} combined with the $C^{1,\alpha}$ spatial regularity, we are able to prove a result which is slightly stronger.

\begin{proposition}\label{p:time-regularity}
Let $s\in(0,1)$, $2< p<2/(1-s)$ and $\alpha\in(0,1)$.
Let $u$ be a weak solution of \eqref{eq:fracplap} in $Q_{1,1}$,
with
\[
\sup_{t\in(-1,0]}
\left(
\|\nabla u(t,\cdot)\|_{L^\infty(B_1)}
+[\nabla u(t,\cdot)]_{C^\alpha(B_1)}
\right)\leq L,
\qquad L\geq1,
\]
and
\[
\sup_{t\in(-1,0]}
{\Tail}_{p-1,sp}(u(t,\cdot)-u(t,0);1)\leq T_0<\infty.
\]
Then, for $t,\tau\in(-1,0]$, $x\in B_{3/4}$ and
$0<|t-\tau|\leq1/8$,
\[
|u(t,x)-u(\tau,x)|
\leq C(L^{p-1}+T_0^{p-1})
\begin{cases}
|t-\tau|,&\sigma<1+\alpha,\\
|t-\tau|(1+|\log |t-\tau||),&\sigma=1+\alpha,\\
|t-\tau|^{(1+\alpha)/\sigma},&\sigma>1+\alpha,
\end{cases}
\]
where $\sigma=2-p(1-s)$ and $C$ depends only on $d,s,p,\alpha$.

If, in addition,
$u\in C(-1,0;L^{p-1}_{sp}(\R^d))$,
then $u$ is continuously differentiable in time in the interior
whenever $\sigma<1+\alpha$. 
\end{proposition}

\begin{proof}
Fix a smooth nonnegative even mollifier $\zeta$ supported in $B_1$, of integral one, and let $\zeta_r(z)=r^{-d}\zeta(z/r)$, with $0<r<1/16$.  For $x\in B_{3/4}$ , set
\[
\psi(z):=\zeta_r(z-x),\qquad
M_r(t):=\int_{\R^d}u(t,z)\psi(z)\dd z.
\]
Since $\zeta$ is even, the linear term in the Taylor expansion cancels and the spatial regularity gives
\begin{equation}\label{eq:spatial_bound}
    |M_r(t)-u(t,x)|\leq CLr^{1+\alpha}.
\end{equation}

Testing the weak equation, we obtain, for almost every $t$,
\[
 M_r'(t)
=-\frac12\iint_{\R^d\times\R^d}
\frac{J_p(u(t,z)-u(t,y))(\psi(z)-\psi(y))}
{|z-y|^{d+sp}}\dd y\dd z.
\]
We split the integral into the three regions
$|z-y|<r$, $r\leq|z-y|<1/8$ and $|z-y|\geq1/8$.

In the first region $|z-y|<r$, whenever $\psi(z)-\psi(y)\neq0$,
both points belong to $B_{2r}(x)$. Thus, using Lemma \ref{l:appendix}, we get 
\[
\left|
J_p(u(t,z)-u(t,y))
-J_p\bigl(\nabla u(t,x)\cdot(z-y)\bigr)
\right|
\leq CL^{p-1}r^\alpha|z-y|^{p-1}.
\]
The integral of the affine term against $\psi(z)-\psi(y)$
vanishes by translation invariance. Using
\[
\|\psi(\cdot+a)-\psi\|_{L^1}\leq C|a|/r,
\]
we bound the first region by
\[
\left|\int_{\R^d}\int_{|z-y|<r}
\frac{J_p(u(t,z)-u(t,y))(\psi(z)-\psi(y))}
{|z-y|^{d+sp}}\dd y\dd z\right|\leq CL^{p-1}r^{\alpha-1}
\int_0^r\rho^{1-\sigma}\dd\rho
\leq CL^{p-1}r^{1+\alpha-\sigma}.
\]

For the second region $r\leq|z-y|<1/8$, antisymmetry allows us to rewrite this contribution by pairing $y$ and $-y$
\[
-\frac12\int_{\R^d}\psi(z)
\int_{r\leq|y|<1/8}
\frac{J_p(u(t,z)-u(t,z+y)) + J_p(u(t,z)-u(t,z-y))}
{|y|^{d+sp}}\dd y\dd z.
\]
Since $J_p$ is an odd function, the principal affine parts cancel each other exactly. Using the spacial $C^{1,
\alpha}$ assumption and Lemma~\ref{l:appendix},
\[
\left|
J_p(u(t,z)-u(t,z+y))
+J_p(u(t,z)-u(t,z-y))
\right|
\leq CL^{p-1}|y|^{p-1+\alpha}.
\]
Consequently, this contribution is bounded by
\[
CL^{p-1}\int_r^{1/8}\rho^{\alpha-\sigma}\dd\rho.
\]

Finally, since $|u(t,x)-u(t,0)|\leq L$ in $B_1$, the tail assumption
gives
\[
\sup_{z\in B_{7/8}}
\int_{|y-z|\geq1/8}
\frac{|J_p(u(t,z)-u(t,y))|}
{|z-y|^{d+sp}}\dd y
\leq C(L^{p-1}+T_0^{p-1}).
\]
Combining the three bounds yields
\[
|M_r'(t)|
\leq C(L^{p-1}+T_0^{p-1})
\begin{cases}
1,&\sigma<1+\alpha,\\
1+|\log r|,&\sigma=1+\alpha,\\
r^{1+\alpha-\sigma},&\sigma>1+\alpha.
\end{cases}
\]
After integration between $t$ and $\tau$, we obtain, for $h=|t-\tau|,$
\[
|u(t,x)-u(\tau,x)|
\leq CLr^{1+\alpha}
+C(L^{p-1}+T_0^{p-1})h
\begin{cases}
1,&\sigma<1+\alpha,\\
1+|\log r|,&\sigma=1+\alpha,\\
r^{1+\alpha-\sigma},&\sigma>1+\alpha.
\end{cases}
\]
For $\sigma<1+\alpha$, let $r\to0$. In the remaining cases,
take $r=h^{1/\sigma}/32$. This proves the estimate.

Suppose now that $\sigma<1+\alpha$ and that $u$ is continuous
in time with values in $L^{p-1}_{sp}(\R^d)$. The symmetrized
local integral defining $\dsp u$ converges uniformly on
compact subsets, since
\[
\int_0^{1/8}\rho^{\alpha-\sigma}\dd\rho<\infty.
\]
The time estimate gives local continuity of $u$, while the
weighted continuity controls the exterior integral.
Therefore $\dsp u$ is continuous in the interior.
The weak equation then gives $\p_tu=-\dsp u$, proving the
last assertion.
\end{proof}

\section*{Acknowledgements}

\noindent D.G. is supported by University of Bologna funds for the project “Attività di collaborazione con università del Nord America” within the framework of “Free Boundaries Problems: Analytical and Numerical Aspects.” D.J. is partially supported by the European Research Council, through the ERC StG project NEW, No.~101220121, and King Abdullah University of Science and Technology (KAUST) under Award No. ORFS-CRG12-2024-6430.
L.S. is supported
by NSF grant DMS-2350263.
These funds supported the visit of D.G. and D.J. to the Department of Mathematics of the University of Chicago, where part of this project was carried out.
D.G. and D.J. are grateful to the Department of Mathematics of the University of Chicago for the warm hospitality.

\subsection*{Use of AI}
This work started shortly after the authors finished the elliptic paper \cite{GJS}. We quickly laid out the outline of the project and started working from the more challenging parts of the proof. We obtained the proof of the first stage (which is the heart of this paper) without any AI assistance. The remaining part of the paper was still technically demanding and took time.

We eventually used ChatGPT 6.0 Astra to help replace our original long and complicated proof of Lemma~\ref{l:centered-time} with a simpler and more elegant argument. Using this proof, we were also able to prove the regularity in time. We also used it to audit, revise and finish up the manuscript. All mathematical arguments were independently checked by the authors.

\bibliographystyle{plain}
\bibliography{bib}

\Addresses
\end{document}